\documentclass[preprint,12pt]{elsarticle}

\usepackage{graphicx}
\usepackage{amssymb}
\usepackage{amsthm}
\usepackage[fleqn,tbtags]{amsmath}
\usepackage{mathtools}
\usepackage{cases}
\usepackage{lipsum}
\usepackage{tikz}
\DeclareMathOperator{\sign}{sign}
\makeatletter
\def\ps@pprintTitle{%
 \let\@oddhead\@empty
 \let\@evenhead\@empty
 \def\@oddfoot{}%
 \let\@evenfoot\@oddfoot}
\makeatother
\biboptions{sort&compress}
\newtheorem{thm}{Theorem}

\newtheorem{lem}[thm]{Lemma}
\theoremstyle{example}

\theoremstyle{definition}

\theoremstyle{remark}
\newtheorem{rem}{Remark}
\journal{}
\title{Sharp subcritical and critical $L^{p}$
Hardy inequalities on the sphere}
\begin{document}

\begin{frontmatter}

\title{\Large Sharp subcritical and critical $L^{p}$
Hardy inequalities on the sphere}

\author{Ahmed A. Abdelhakim}
\address{Mathematics Department, Faculty of Science, Assiut University, Assiut, Egypt}
\ead{ahmed.abdelhakim@aun.edu.eg}

\begin{abstract}
We prove sharp inequalities of Hardy type
for functions in the Sobolev space
$W^{1,p}$ on the unit sphere $\mathbb{S}^{n-1}$ in $\mathbb{R}^{n}$. We achieve this in both the subcritical and critical cases. The method we use to show optimality takes into account all the constants involved in our inequalities.
\end{abstract}


\end{frontmatter}
\section{Introduction}
\indent Since the authors in \cite{Kombe} introduced an idea
to obtain sharp inequalities of Hardy and Rellich type on noncompact Riemannian manifolds, there has been interest in establishing the corresponding inequalities on the sphere.
Obviously, the method given in \cite{Kombe} does not apply directly to the compact manifold $\mathbb{S}^{n}$. Another technical difficulty comes from the fact that
the Laplacian of the geodesic distance on the sphere changes sign. Xiao \cite{Xiao} was the first to make progress on this problem. He obtained $L^{2}$
inequalities of the Hardy type on the sphere $\mathbb{S}^{n}$, $n\geq 3$. These results were complemented in \cite{ahmed} in the limiting case
where optimal $L^{2}$ inequalities of the Hardy type were proved on $\mathbb{S}^{2}$. Xiao's results were also extended to $L^{p}(\mathbb{S}^{n})$, $1<p<n$, $n\geq 3$ in \cite{xsun}.
\\
\indent
In \cite{ahmed,xsun,Xiao}, the singularity is assumed to be at either the north or south pole so that
the geodesic distance will be simply the
polar angle. So, if the singularity is not polar, we must rotate the local axes in order to apply these inequalities. But we should not need to rotate the axes. It is not physically plausible as we could be dealing with a punctured sphere missing a closed connected piece, or a sphere with a crack missing an open simple curve. This motivates us to look for $L^{p}$ Hardy inequalities in which the singularity is the geodesic distance from an arbitrary point.\\
\indent The general geodesic distance is very recently considered in \cite{Abimbola1,Songting}. The proofs in \cite{Abimbola1,Songting} are based on a formula for the Laplacian of the geodesic distance. Inconveniently, no reference was provided for that formula, and no proof of it was given either. More inconveniently, the definition of the geodesic distance on $\mathbb{S}^{n}$ adopted in \cite{Abimbola1,Songting} is not specified. Such definition is
important to understand the set up of the inequalities.
This is also technically important since the singularities in the inequalities involve trigonometric functions. That in turn necessitates determining whether the range of the geodesic distance is $[-\frac{\pi}{2},\frac{\pi}{2}]$ or $[0,\pi]$.\\
\indent In \cite{Songting}, an $L^{2}$ Hardy inequality is proved in high dimensions using Xiao's method.
The results in \cite{Abimbola1} are supposed to generalize the $L^{2}$ Hardy inequality presented in
\cite{Songting} to an $L^{p}$ inequality on
$\mathbb{S}^{n}$ where $1<p<n$ and $n\geq 3$.
Unfortunately, the proof presented in \cite{Abimbola1}
requires revision.  We discuss that in detail in \cite{ahmed3}, where we additionally prove limiting case $L^{n}$ Hardy type inequalities on the sphere
$\mathbb{S}^{n}$, $n\geq 2$, with optimal coefficients considering the general geodesic distance and adopting Xiao's method.\\
\indent
When it comes to the sharpness of the coefficients,
all the results in \cite{Abimbola1,ahmed,xsun,Xiao, Songting} are based on the same principle that we find insufficient. The method implemented is also unnecessarily involved at times. Inequalities of Hardy type obtained in \cite{Abimbola1,ahmed,xsun,Xiao, Songting} on $\mathbb{S}^{n}$ take the generic form
\begin{equation*}
A_{n,p}\int_{\mathbb{S}^{n}}
\frac{|u|^{p}}{|f(\rho)|^{p}}d\sigma_{n}\leq
B_{n,p}\int_{\mathbb{S}^{n}}
|\nabla u|^{p}d\sigma_{n}+
C_{n,p}\int_{\mathbb{S}^{n}}
\frac{|u|^{p}}{|f(\rho)|^{p-2}}d\sigma_{n},\quad 2\leq p\leq n.
\end{equation*}
where $u\in C^{\infty}(\mathbb{S}^{n})$, and $f$
is a continuous function of the geodesic distance $\rho$. Sharpness of the constants $A_{n,p}$,
$B_{n,p}$ and $C_{n,p}$ is claimed to be proved by showing that
\begin{equation}\label{shrp1}
\sup_{u\in C^{\infty}(\mathbb{S}^{n})\setminus \left\{0\right\}}{\frac{A_{n,p}\int_{\mathbb{S}^{n}}
\frac{|u|^{p}}{|f(\rho)|^{p}}d\sigma_{n}}
{B_{n,p}\int_{\mathbb{S}^{n}}
|\nabla u|^{p}d\sigma_{n}+
C_{n,p}\int_{\mathbb{S}^{n}}
\frac{|u|^{p}}{|f(\rho)|^{p-2}}d\sigma_{n}}}=1.
\end{equation}
But the latter does not prove that the constants
$B_{n,p}$ and $C_{n,p}$ are both the smallest possible.\\
\indent We prove sharp $L^{p}$ Hardy inequalities
on the sphere $\mathbb{S}^{n}$ in $\mathbb{R}^{n+1}$
in both the subcritical and critical exponent cases.
We follow a method of proof different from that
used in \cite{Abimbola1,ahmed,xsun,Xiao, Songting}. The method we adopt is fairly simpler and require less computations. Before delving into the derivation of the inequalities, we use explicit formulas for the geodesic distance, the surface gradient and the Laplace-Beltrami operator on the $n$-dimensional sphere to demonstrate
some basic properties of the geodesic distance on which we rely heavily in obtaining our results.\\
\indent Besides proving (\ref{shrp1}), we show the optimality of all the constants in our inequalities by proving that
\begin{equation*}
\begin{split}
&\sup_{u\in C^{\infty}(\mathbb{S}^{n})\setminus \left\{0\right\}}{\frac{A_{n,p}\int_{\mathbb{S}^{n}}
\frac{|u|^{p}}{|f(\rho)|^{p}}d\sigma_{n}-
C_{n,p}\int_{\mathbb{S}^{n}}
\frac{|u|^{p}}{|f(\rho)|^{p-2}}d\sigma_{n}}
{B_{n,p}\int_{\mathbb{S}^{n}}
|\nabla u|^{p}d\sigma_{n}}}=1,\\
&\sup_{u\in C^{\infty}(\mathbb{S}^{n})\setminus \left\{0\right\}}{\frac{A_{n,p}\int_{\mathbb{S}^{n}}
\frac{|u|^{p}}{|f(\rho)|^{p}}d\sigma_{n}-
B_{n,p}\int_{\mathbb{S}^{n}}
|\nabla u|^{p}d\sigma_{n}}
{C_{n,p}\int_{\mathbb{S}^{n}}
\frac{|u|^{p}}{|f(\rho)|^{p-2}}d\sigma_{n}}}=1.
\end{split}
\end{equation*}
To achieve this, we exploit a formula for integration over spheres (see (\ref{frml}) below) to calculate the ratios above for explicit functions in the appropriate Sobolev space.
\section{Preliminaries}
Let $n\geq 2$ and define $\Theta_{n-1}:=(\theta_j)_{j=1}^{n-1}\in
[0,\pi]^{n-2}\times[0,2\pi]$.
Then any point on the unit sphere
$\mathbb{S}^{n-1}$ in $\mathbb{R}^{n}$ has the spherical coordinates parametrization $\left(x_{m}(\Theta_{n-1})\right)_{m=1}^{n}$, where
\begin{equation}\label{xm}
x_{m}(\Theta_{n-1}):=\left\{
                \begin{array}{ll}
                  \cos{\theta_1}, & \hbox{$m=1$;} \vspace{0.1 cm} \\
                  \prod_{j=1}^{m-1}\sin{\theta_j}\cos{\theta_m}, & \hbox{$2\leq m\leq n-1$;} \vspace{0.15 cm}\\
                  \prod_{j=1}^{n-1}\sin{\theta_j}, & \hbox{$m=n$.}
                \end{array}
              \right.
\end{equation}
The gradient $\nabla_{\mathbb{S}^{n-1}}$ on the sphere $\mathbb{S}^{n-1}$  is then given by
\begin{equation*}
\nabla_{\mathbb{S}^{n-1}}  = \frac{\partial }{\partial \theta_1} \widehat{{\theta}_{1}} + \frac{1}{\sin \theta_1} \frac{\partial }{\partial \theta_2} \widehat{{\theta}_{2}}  + \cdots + \frac{1}{\sin \theta_1 \cdots \sin \theta_{n-2}} \frac{\partial }{\partial \theta_{n-1}} \widehat{{\theta}_{n-1}},
\end{equation*}
where $\left\{\widehat{{\theta}_{j}}\right\}$
is an orthonormal set of tangential vectors with
$\widehat{{\theta}_{j}}$ pointing in the direction of increase of ${\theta}_{j}$. Moreover, the Laplace-Beltrami operator $\Delta_{\mathbb{S}^{n-1}}$ is given by
\begin{align}\label{lapbelt}
\nonumber \hspace{-1 cm}\Delta_{\mathbb{S}^{n-1}}  =&\,
 \frac{1}{\sin^{n-2}{\theta_{1}}}
 \frac{\partial }{\partial \theta_1}
 \left(\sin^{n-2}{\theta_{1}} \frac{\partial }{\partial \theta_1} \right)+
\frac{1}{\sin^{2}{\theta_{1}}\sin^{n-3}{\theta_{2}}}
 \frac{\partial }{\partial \theta_2}
 \left(\sin^{n-3}{\theta_{2}} \frac{\partial }{\partial \theta_2} \right)+\\&\nonumber
+\dots+\frac{1}{\sin^{2}{\theta_{1}}\sin^{2}{\theta_{2}}...
 \sin^{2}{\theta_{n-3}}\sin{\theta_{n-2}}}
 \frac{\partial }{\partial \theta_{n-2}}
 \left(\sin{\theta_{n-2}} \frac{\partial }{\partial \theta_{n-2}} \right)+\\&+
  \frac{1}{\sin^{2}{\theta_{1}}\sin^{2}{\theta_{2}}...
 \sin^{2}{\theta_{n-2}}}
 \frac{\partial^2 }{\partial \theta^{2}_{n-1}}.
\end{align}
Identifying each point $(x_m(\Theta_{n-1}))_{m=1}^{n}\in \mathbb{S}^{n-1}$ with its parameters $\Theta_{n-1}$, we can express the geodesic distance $d(\Theta_{n-1},\Phi_{n-1})$ from a point $\Phi_{n-1}
\in \mathbb{S}^{n-1}$ as
\begin{equation}\label{gdsc}
d(\Theta_{n-1},\Phi_{n-1})=\arccos{
\lambda(\Theta_{n-1},\Phi_{n-1})},
\end{equation}
where
\begin{equation}\label{lmda}
\lambda(\Theta_{n-1},\Phi_{n-1}):=\sum_{m=1}^{n}
x_m(\Theta_{n-1})x_m(\Phi_{n-1}).
\end{equation}
\subsection{A useful formula for integration over
$\mathbb{S}^{n}$}
Let $v\in \mathbb{R}^{n}\setminus\left\{0\right\}$
and let $F\in L^{1}(\mathbb{S}^{n-1}\longrightarrow \mathbb{R})$ be such that $F(\Theta_{n-1}):=f(v\cdot\Theta_{n-1})$. Then
\begin{equation}\label{frml}
\int_{\mathbb{S}^{n-1}}
F(\Theta_{n-1})\,d\sigma_{n-1}=C_{n}
\int_{-1}^{1}f(|v|t)\left(\sqrt{1-t^2}\right)^{n-3}dt,
\end{equation}
where $\displaystyle C_{n}=\frac{
2\pi^\frac{n-1}{2}}{\Gamma{\left(\frac{n-1}{2}\right)}}$.
(See \cite{loukas}, Appendix D).
\subsection{The Sobolev space
$W^{1,p}\left(\mathbb{S}^{n-1}\right)$}
It is useful to define the weak
Laplace - Beltrami gradient
of a function $f\in L^{1}\left(\mathbb{S}^{n-1}\right)$. Let $f\in C^{\infty}\left(\mathbb{S}^{n-1}\rightarrow \mathbb{R}\right)$.
Then, by the divergence theorem, we have
\begin{equation*}
\int_{\mathbb{S}^{2}}\nabla f \cdot V \, d\sigma
=-\int_{\mathbb{S}^{2}} f\, \nabla\cdot V \, d\sigma
\end{equation*}
for any vector field
$V\in C^{\infty}\left(\mathbb{S}^{n-1}\rightarrow
T\left(\mathbb{S}^{n-1}\right)\right)$ where $T\left(\mathbb{S}^{n-1}\right)$
is the tangent bundle on the smooth manifold $\mathbb{S}^{n-1}$. Therefore, $f$ is weakly differentiable if there exists
a vector field $\Psi_{f}\in L^{1}\left(\mathbb{S}^{n-1}\rightarrow T\left(\mathbb{S}^{n-1}\right)\right)$ such that
\begin{equation*}
\int_{\mathbb{S}^{n-1}} \Psi_{f} \cdot V \, d\sigma_{n-1}
=-\int_{\mathbb{S}^{n-1}} f\, \nabla\cdot V \, d\sigma_{n-1},
\quad \forall\, V\in C^{\infty}\left(\mathbb{S}^{n-1}\rightarrow
T\left(\mathbb{S}^{n-1}\right)\right).
\end{equation*}
Such vector field $\Psi_{f}$, if it exists, is called the weak surface gradient of $f$. The weak surface gradient is unique up to a set of measure zero.
As shown in (\cite{Eichhorn}, Proposition 3.2., page 15)
\begin{equation}\label{dfnw1p}
W^{1,p}\left(\mathbb{S}^{n-1}\right):=
\left\{f\in L^{p}(\mathbb{S}^{n-1}):
|\Psi_{f}|\in L^{p}\left(\mathbb{S}^{n-1}\right)\right\}.
\end{equation}
The definition (\ref{dfnw1p}) is equivalent to defining $W^{1,p}\left(\mathbb{S}^{n-1}\right)$ as the completion of the space $C^{\infty}\left(\mathbb{S}^{n-1}\right)$
in the usual Sobolev norm.\\
\subsection{A formula for integration on the sphere}
Let $v\in \mathbb{R}^{n}\setminus\left\{0\right\}$
and let $F\in L^{1}(\mathbb{S}^{n-1}\longrightarrow \mathbb{R})$ be such that $F(\Theta_{n-1}):=f(v\cdot\Theta_{n-1})$. Then
\begin{equation}\label{frml}
\int_{\mathbb{S}^{n-1}}
F(\Theta_{n-1})\,d\sigma_{n-1}=C_{n}
\int_{-1}^{1}f(|v|t)\left(\sqrt{1-t^2}\right)^{n-3}dt,
\end{equation}
where $\displaystyle C_{n}=\frac{
2\pi^\frac{n-1}{2}}{\Gamma{\left(\frac{n-1}{2}\right)}}$.\\
\indent
In the next section, we show interesting properties of the geodesic distance on the sphere that carry on to all dimensions.\\
\section{The gradient and Laplacian of the geodesic distance on the sphere}
The geodesic distance $d$ on the sphere $\mathbb{S}^{n-1}$ has a gradient and Laplacian analogous to those of the Euclidean metric. We demonstrate that $|\nabla_{\mathbb{S}^{n-1}}d|_{\mathbb{S}^{n-1}}=1$ and that $\Delta_{\mathbb{S}^{n-1}}d=(n-2)\cos{d}/\sin{d}$, in any dimension $n\geq 2$. Unlike with the Euclidean distance, the laplacian of the geodesic
distance $d$ changes sign on the sphere. We start with showing that
\begin{equation*}
x_j(\Theta_{n-1})x_k(\Theta_{n-1})+
\nabla_{\mathbb{S}^{n-1}}x_j(\Theta_{n-1})\cdot
\nabla_{\mathbb{S}^{n-1}}x_k(\Theta_{n-1})=\delta_{jk},
\end{equation*}
the kronecker delta.
\begin{lem}\label{lemkronecker}
Let $n\geq 2$ and let $\nabla_{\mathbb{S}^{n-1}}$ be the gradient on the unit sphere $\mathbb{S}^{n-1}$ in
$\mathbb{R}^{n}$. Then
\begin{equation}\label{nabla1}
x^2_m(\Theta_{n-1})+|\nabla_{\mathbb{S}^{n-1}}x_m(\Theta_{n-1})|^2=1,
\end{equation}
\begin{equation}\label{nabla2}
x_{\ell}(\Theta_{n-1})x_m(\Theta_{n-1})+
\nabla_{\mathbb{S}^{n-1}}x_{\ell}(\Theta_{n-1})\cdot
\nabla_{\mathbb{S}^{n-1}}x_m(\Theta_{n-1})=0, \quad {\ell}\neq m.
\end{equation}
\end{lem}
\begin{proof}
Lemma \ref{lemkronecker} is
trivial in the dimension $n=2$ and similarly easily verifiable when $n=3$ by the computation
\begin{equation*}
\nabla_{\mathbb{S}^{2}}x_m(\Theta_{2})=
\left\{
\begin{array}{ll}
-\sin{\theta_1}\widehat{{\theta}_{1}}, & \hbox{$m=1$;}\\
\cos{\theta_{1}}\cos{\theta_{2}}\widehat{{\theta}_{1}}
-\sin{\theta_{2}}\widehat{{\theta}_{2}}, & \hbox{$m=2$;} \vspace{0.15 cm}\\
\cos{\theta_{1}}\sin{\theta_{2}}\widehat{{\theta}_{1}}
+\cos{\theta_{2}}\widehat{{\theta}_{2}}, & \hbox{$m=3$}.
                \end{array}
              \right.
\end{equation*}
\indent Suppose $n\geq 4$. Again, the identity (\ref{nabla1})
is easy to prove when $m=1,2$, and so is
the identity (\ref{nabla2}) when $1\leq \ell,m\leq 2$.
Observe that, for all $n\geq 4$,
\begin{equation*}
x_m(\Theta_{n-1})=
x_m(\Theta_{2}),\quad
\nabla_{\mathbb{S}^{n-1}} x_m(\Theta_{n-1})=
\nabla_{\mathbb{S}^{2}}x_m(\Theta_{2}),\quad m=1,2.
\end{equation*}
\indent Fix $m\geq 3$. We get
(\ref{nabla1}) from the calculation
\begin{align}
&\hspace{-1cm}\nonumber
\nabla_{\mathbb{S}^{n-1}}x_m(\Theta_{n-1})
\label{nabla0}=
\left\{
\begin{array}{ll}
\sum_{j=1}^{m-2}\cos{\theta_j}
\prod_{k=j+1}^{m-1}
\sin{\theta_k}\cos{\theta_m}\widehat{{\theta}_{j}}+ & \hbox{} \vspace{0.15 cm}\\
+\cos{\theta_{m-1}}\cos{\theta_m}\widehat{{\theta}_{m-1}}-
\sin{\theta_m}\widehat{{\theta}_{m}}, & \hbox{$3\leq m\leq n-1$;} \vspace{0.15 cm}\\
\sum_{j=1}^{n-2}\cos{\theta_j}
\prod_{k=j+1}^{n-1}
\sin{\theta_k}\widehat{{\theta}_{j}}+
\cos{\theta_{n-1}}\widehat{{\theta}_{n-1}}, & \hbox{$m=n$},
                \end{array}
              \right.
\end{align}
and the orthonormality of the set $\left\{\widehat{{\theta}_{j}}\right\}_{j=1}^{n-1}$
along with the identity
\begin{equation}\label{sumeq1}
U(\Theta_{m-1}):=\sum_{j=1}^{m}x^{2}_{j}
\left(\theta_{m-1},...,\theta_{1} \right)=1.
%\prod_{j=1}^{m-1}\sin^2{\theta_j}+
%\sum_{j=1}^{m-2}\cos^2{\theta_j}
%\prod_{k=j+1}^{m-1}\sin^2{\theta_k}+
%\cos^2{\theta_{m-1}}=1.
\end{equation}
Indeed, one can write
\begin{equation*}
x^2_m(\Theta_{n-1})+
|\nabla_{\mathbb{S}^{n-1}}x_m(\Theta_{n-1})|^2=
\left\{
  \begin{array}{ll}
    U(\Theta_{m-1})\cos^2{\theta_m}+
\sin^2{\theta_{m}}, & \hbox{$m\leq n-1$;} \\
    U(\Theta_{n-1}), & \hbox{$m=n$.}
  \end{array}
\right.
\end{equation*}
Now, we turn to the identity (\ref{nabla2}). Assume, losing no generality, that $1\leq\ell<m$.
Then, tedious yet straightforward computation uncovers that
\begin{eqnarray*}
-\nabla_{\mathbb{S}^{n-1}}x_{1}(\Theta_{n-1})\cdot
\nabla_{\mathbb{S}^{n-1}}x_m(\Theta_{n-1})&=& \left\{
  \begin{array}{ll}
\cos{\theta_{1}}\prod_{j=1}^{m-1}\sin{\theta_{j}}
\cos{\theta_{m}}, & \hbox{$3\leq m\leq n-1$;} \vspace{0.12 cm} \\
\cos{\theta_{1}}
\prod_{j=1}^{n-1}\sin{\theta_{j}}, & \hbox{$m=n$.}
  \end{array}
\right.\vspace{0.12 cm} \\
&=&x_{1}(\Theta_{n-1})x_{m}(\Theta_{n-1}),
\end{eqnarray*}
and when $2\leq \ell\leq m-1$ we have
\begin{eqnarray*}
&&\nabla_{\mathbb{S}^{n-1}}x_{\ell}(\Theta_{n-1})\cdot
\nabla_{\mathbb{S}^{n-1}}x_m(\Theta_{n-1})\\ &=&
\left(\sum_{j=1}^{\ell-2}
\cos^{2}{\theta_{j}} \prod_{j=k+1}^{\ell-1}\sin^{2}{\theta_{k}}
-\sin^{2}{\theta_{\ell-1}}
\right)\cos{\theta_{\ell}}
\left\{ \begin{array}{ll}
\prod_{j=\ell}^{m-1}\sin{\theta_{j}}\cos{\theta_{m}}, & \hbox{$m\leq n-1$;}\vspace{0.12 cm} \\
\prod_{j=\ell}^{n-1}\sin{\theta_{j}}, & \hbox{$m=n$}
  \end{array}
\right.\vspace{0.12 cm} \\
\vspace{0.12 cm} \\ &=&
\prod_{j=1}^{\ell-1}\sin^{2}{\theta_{j}}
\cos{\theta_{\ell}}
\left\{
  \begin{array}{ll}
\prod_{j=\ell}^{m-1}\sin{\theta_{j}}\cos{\theta_{m}}, & \hbox{$m\leq n-1$;}\vspace{0.12 cm} \\
\prod_{j=\ell}^{n-1}\sin{\theta_{j}}, & \hbox{$m=n$}
  \end{array}
\right.\\
&=&-x_{\ell}(\Theta_{n-1})x_{m}(\Theta_{n-1}).
\end{eqnarray*}
\end{proof}
\indent The next lemma shows that the components
$x_{m}$
are eigenfunctions of the Laplace - Beltrami operator
(\ref{lapbelt}):
\begin{lem}\label{deltasnlm}
\begin{equation}\label{deltasn}
\Delta_{\mathbb{S}^{n-1}}x_{m}(\Theta_{n-1})
=-(n-1)x_{m}(\Theta_{n-1}),\; n\geq 2.
\end{equation}
\end{lem}
\begin{proof}
Write
\begin{equation*}
\Delta_{\mathbb{S}^{n-1}}=
\sum_{\ell=1}^{n-1}\Delta_{\ell},
\end{equation*}
where $\Delta_{\ell}$ are the differential operators
\begin{eqnarray*}
\Delta_{1}&:=&\frac{1}{\sin^{n-2}{\theta_{1}}}
 \frac{\partial }{\partial \theta_1}
 \left(\sin^{n-2}{\theta_{1}} \frac{\partial }{\partial \theta_1} \right), \\
\Delta_{\ell}&:=&\frac{1}{\prod_{j=1}^{\ell-1}
\sin^{2}{\theta_{j}}}
\frac{1}{\sin^{n-\ell-1}{\theta_{\ell}}}
\frac{\partial}{\partial \theta_{\ell}}\left(\sin^{n-\ell-1}{\theta_{\ell}}
\frac{\partial}{\partial \theta_{\ell}}\right),\;2\leq \ell\leq n-1.
\end{eqnarray*}
Then, to prove (\ref{deltasn}), it suffices to establish
that
\begin{eqnarray}
\label{suffest} \sum_{\ell=1}^{m}
\Delta_{\ell}x_{m}(\Theta_{n-1})&=&
-(n-1)x_{m}(\Theta_{n-1}),\;1\leq m\leq n-1,\\
\label{suffestn} \sum_{\ell=1}^{n-1}
\Delta_{\ell}x_{n}(\Theta_{n-1})&=&
-(n-1)x_{n}(\Theta_{n-1}).
\end{eqnarray}
Straightforward calculations affirm (\ref{suffest}) when $m=1$. We prove (\ref{suffest}) by induction. Assume (\ref{suffest}) holds true for some $1\leq m\leq n-2$. Let us define
\begin{equation*}
\delta_{m}(\Theta_{n-1}):=
\frac{x_{m+1}(\Theta_{n-1})}{x_{m}(\Theta_{n-1})}=
\frac{\sin{\theta_{m}}\cos{\theta_{m+1}}}{\cos{\theta_{m}}}.
\end{equation*}
Now, since
\begin{eqnarray*}
\sum_{\ell=1}^{m-1}
\Delta_{\ell}x_{m+1}
&=&\delta_{m}\sum_{\ell=1}^{m-1}
\Delta_{\ell}x_{m}\\
&=&\delta_{m}\sum_{\ell=1}^{m}
\Delta_{\ell}x_{m}-\delta_{m}\Delta_{m}x_{m}\\
&=&-(n-1)x_{m+1}-\delta_{m}\Delta_{m}x_{m},
\end{eqnarray*}
then we have
\begin{equation*}
\sum_{\ell=1}^{m+1}
\Delta_{\ell}x_{m+1}
=-(n-1)x_{m+1}-\delta_{m}\Delta_{m}x_{m}
+\left(\Delta_{m}+\Delta_{m+1}\right)
x_{m+1}.
\end{equation*}
Consequently, what remains to prove is
 \begin{equation}\label{rem1}
-\delta_{m}\Delta_{m}x_{m}
+\left(\Delta_{m}+\Delta_{m+1}\right)
x_{m+1}=0.
 \end{equation}
Calculating further, we find
\begin{eqnarray*}
\Delta_{m}x_{m+1}&=&
\delta_{m}\Delta_{m}x_{m}+x_{m}\Delta_{m}\delta_{m}+
\frac{2}{\prod_{j=1}^{m-1}\sin^{2}{\theta_{j}}}
\frac{\partial x_{m}}{\partial \theta_{m}}
\frac{\partial \delta_{m}}{\partial \theta_{m}},\\
\Delta_{m+1}x_{m+1}&=&
x_{m}\Delta_{m+1}\delta_{m}.
\end{eqnarray*}
Therefore, (\ref{rem1}) is equivalent to
 \begin{equation}\label{rem2}
\left(\Delta_{m}+\Delta_{m+1}\right)\delta_{m}=
\frac{2\sin{\theta_{m}}}{\prod_{j=1}^{m-1}
\sin^{2}{\theta_{j}}\cos{\theta_{m}}}\frac{\partial \delta_{m}}{\partial \theta_{m}}
 \end{equation}
which is easy to verify. Having proved (\ref{suffest}),
we can exploit its validity for
$m=n-1$ in particular to prove (\ref{suffestn}). Write $x_{n}=x_{n-1}\delta_{n-1}$ with $\delta_{n-1}(\Theta_{n-1}):=
\sin{\theta_{n-1}}/\cos{\theta_{n-1}}$.
Arguing as above, we discover that
\begin{equation*}
\sum_{\ell=1}^{n-1}
\Delta_{\ell}x_{n}
=-(n-1)x_{n}+x_{n-1}\Delta_{n-1}\delta_{n-1}+
\frac{2}{\prod_{j=1}^{n-2}\sin^{2}{\theta_{j}}}
\frac{\partial x_{n-1}}{\partial \theta_{n-1}}
\frac{\partial \delta_{n-1}}{\partial \theta_{n-1}}.
\end{equation*}
This reduces (\ref{suffestn}) to
\begin{equation*}
\frac{\partial^{2} \delta_{n-1}}{\partial \theta^{2}_{n-1}}=2\frac{\sin{\theta_{n-1}}}{
\cos{\theta_{n-1}}}\frac{\partial \delta_{n-1}}{\partial \theta_{n-1}}
\end{equation*}
which is simple to check.
\end{proof}
\begin{lem}\label{lemlamda}
Let $\Phi_{n-1}\in \mathbb{S}^{n-1}$, and
let $\,\lambda(.,\Phi_{n-1}):
\mathbb{S}^{n-1}\rightarrow [-1,1]\,$ be the function defined in (\ref{lmda}).
Then
\begin{eqnarray}
\label{nabllmda}\left|\nabla_{\mathbb{S}^{n-1}}\lambda\right|&=&
\sqrt{1-\lambda^{2}},\\
\label{deltlmda}\Delta_{\mathbb{S}^{n-1}}\lambda&=&-(n-1)
\lambda.
\end{eqnarray}
\end{lem}
\begin{proof}
Using Lemma \ref{lemkronecker}, we obtain
\begin{flalign*}
\begin{aligned}
\left|\nabla_{\mathbb{S}^{n-1}}
\lambda(\Theta_{n-1},\Phi_{n-1})\right|^2=&\sum_{i}
x^2_i(\Phi_{n-1})|\nabla_{\mathbb{S}^{n-1}}
x_i(\Theta_{n-1})|^2\\
&+2
\sum_{i\neq j}x_i(\Phi_{n-1})x_j(\Phi_{n-1})
\nabla_{\mathbb{S}^{n-1}} x_i(\Theta_{n-1})\cdot\nabla_{\mathbb{S}^{n-1}}
x_j(\Theta_{n-1})\\
=&\sum_{i}
x^2_i(\Phi_{n-1})-\sum_{i}
x^2_i(\Phi_{n-1})x^2_i(\Theta_{n-1})\\
&-2
\sum_{i\neq j}x_i(\Phi_{n-1})
x_i(\Theta_{n-1})x_j(\Phi_{n-1})
x_j(\Theta_{n-1})\\
=&1-\lambda^{2}(\Theta_{n-1},\Phi_{n-1}).
\end{aligned}&&&
\end{flalign*}
This shows (\ref{nabllmda}). We also get (\ref{deltlmda}) as a direct consequence of Lemma \ref{deltasnlm}, since
$\lambda(\Theta_{n-1},\Phi_{n-1})$ is a linear
combination of eigenfunctions of $\Delta_{\mathbb{S}^{n-1}}$ that all correspond
to the eigenvalue $-(n-1)$.
\end{proof}
\begin{lem}\label{gradlap}
Let $\Phi_{n-1}$ be a point on the sphere
$\mathbb{S}^{n-1}$, and let $\,d(.,\Phi_{n-1}):
\mathbb{S}^{n-1}\rightarrow [0,\pi]\,$ be the geodesic distance from $\Phi_{n-1}$ on $\mathbb{S}^{n-1}$
defined in (\ref{gdsc}).
Then
\begin{eqnarray}
\label{nablagdsc}
\left|\nabla_{\mathbb{S}^{n-1}}d\right|&=&
1,\\
\label{deltagdsc}\Delta_{\mathbb{S}^{n-1}}d
&=&(n-2)
\frac{\cos{d}}{\sin{d}}.
\end{eqnarray}
\end{lem}
\begin{proof}
From (\ref{gdsc}), we find
\begin{equation}\label{tkdv}
\nabla_{\mathbb{S}^{n-1}}d(\Theta_{n-1},\Phi_{n-1})=
-\frac{\nabla_{\mathbb{S}^{n-1}}
\lambda(\Theta_{n-1},\Phi_{n-1})}{\sqrt{1-
\lambda^{2}(\Theta_{n-1},\Phi_{n-1})}}.
\end{equation}
Hence (\ref{nablagdsc}) follows from (\ref{nabllmda}).
Taking the divergence of both sides of (\ref{tkdv}),
then substituting for $\nabla_{\mathbb{S}^{n-1}}\lambda$ from (\ref{nabllmda}) and for
$\Delta_{\mathbb{S}^{n-1}}\lambda$ from
(\ref{deltlmda}), we deduce that
\begin{equation*}
\begin{split}
\Delta_{\mathbb{S}^{n-1}}d(\Theta_{n-1},\Phi_{n-1})=&
-\frac{\Delta_{\mathbb{S}^{n-1}}
\lambda(\Theta_{n-1},\Phi_{n-1})}{\sqrt{1-
\lambda^{2}(\Theta_{n-1},\Phi_{n-1})}}\\&-
\frac{\lambda(\Theta_{n-1},\Phi_{n-1})
|\nabla_{\mathbb{S}^{n-1}}
\lambda(\Theta_{n-1},\Phi_{n-1})|^2}{\left(1-
\lambda^{2}(\Theta_{n-1},\Phi_{n-1})\right)^{\frac{3}{2}}}\\
=&
\frac{(n-2)
\lambda(\Theta_{n-1},\Phi_{n-1})}{\sqrt{1-
\lambda^{2}(\Theta_{n-1},\Phi_{n-1})}},
\end{split}
\end{equation*}
which yields (\ref{deltagdsc}) in the light of (\ref{gdsc}).
\end{proof}
\section{Subcritical $L^{p}$ Hardy inequalities}
Let $\mathbb{S}^{n-1}$ be the unit sphere in
$\mathbb{R}^{n}$, $n\geq 4$. Let $1< p<n-1$ and consider the following nonlinear positive functionals on $W^{1,p}(\mathbb{S}^{n-1}\longrightarrow \mathbb{R})$:
\begin{eqnarray*}
S_{p}(u) &:=& \int_{\mathbb{S}^{n-1}}
\frac{|u|^{p}}{\sin^{p}{d}}d\sigma_{n-1}, \\
\widetilde{S}_{p}(u) &:=& \int_{\mathbb{S}^{n-1}}
\frac{|u|^{p}}{\sin^{p-2}{d}}d\sigma_{n-1}, \\
  T_{p}(u) &:=& \int_{\mathbb{S}^{n-1}}
\frac{|u|^{p}}{|\tan{d}|^{p}}d\sigma_{n-1}, \\
  \widetilde{T}_{p}(u) &:=& \int_{\mathbb{S}^{n-1}}
\frac{|u|^{p}}{|\tan{d}|^{p-2}}d\sigma_{n-1}, \\
F_{p}(u) &:=& \int_{\mathbb{S}^{n-1}}
|\nabla_{\mathbb{S}^{n-1}} u|^{p}d\sigma_{n-1}, \\
G_{p}(u) &:=& \int_{\mathbb{S}^{n-1}}
|\nabla_{\mathbb{S}^{n-1}} u|^{p}|\cos{d}|^{p} d\sigma_{n-1}.
\end{eqnarray*}
Define also the constant
\begin{equation*}
\alpha_{n,p}:=\frac{n-p-1}{p}.
\end{equation*}
\begin{rem}\label{rmcon}
Formula (\ref{frml}) makes it clear that the integrals
$\displaystyle \int_{\mathbb{S}^{n-1}}
\frac{|u|^{p}}
{|\tan{d}|^{p}}d\sigma_{n-1}$,
$\;\displaystyle\int_{\mathbb{S}^{n-1}}
\frac{|u|^{p}}
{\sin^{p}{d}}d\sigma_{n-1}$ are convergent
when $u$ is continuous. Indeed,
recalling that $d(\Theta_{n-1},\Phi_{n-1})=\arccos{\left(
\Theta_{n-1}\cdot\Phi_{n-1}\right)}$, $\Phi_{n-1}
\in \mathbb{S}^{n-1}$, we immediately see
\begin{equation*}
\begin{split}
&\int_{\mathbb{S}^{n-1}}
\frac{d\sigma_{n-1}}
{|\tan{d}|^{p}}<
\int_{\mathbb{S}^{n-1}}
\frac{d\sigma_{n-1}}
{\sin^{p}{d}}
=\int_{\mathbb{S}^{n-1}}
\frac{d\sigma_{n-1}}
{\left(1-\left(\Theta_{n-1}\cdot\Phi_{n-1}
\right)^{2}\right)^{p/2}}\\
&=2C_{n}
\int_{0}^{1}\left(1-t^2\right)^{\frac{
n-p-3}{2}}dt
\end{split}
\end{equation*}
which exists for $p<n-1$.
\end{rem}
We show that the functionals $T_{p}$, $\widetilde{T}_{p}$, $S_{p}$, $\widetilde{S}_{p}$
are all well-defined and related by the following
$L^{p}$ inequalities of Hardy type:
 \begin{thm}(Subcritical $L^{p}$ Hardy inequalities)\\
Suppose $u\in W^{1,p}(\mathbb{S}^{n-1}\longrightarrow \mathbb{R})$, $n\geq 4$. Then
$\frac{u}{\sin{d}}\in L^{p}(\mathbb{S}^{n-1})$,
when $1<p<n-1$, and $\frac{u}{|\tan{d}|}\in L^{p}(\mathbb{S}^{n-1})$ when $2\leq p<n-1$.
Moreover
\begin{eqnarray}
\label{f1} \alpha^{p}_{n,p}S_{p}(u)&\leq&G_{p}(u) +
(n-p)\alpha^{p-1}_{n,p}\widetilde{S}_{p}(u),\quad
1<p<n-1,\\
\label{f3}
\alpha^{p}_{n,p}T_{p}(u)&\leq&F_{p}(u) +
(p-1)\alpha^{p-1}_{n,p}\widetilde{T}_{p}(u),\quad
2\leq p<n-1.
\end{eqnarray}
\end{thm}
\begin{proof}
Let us start with the inequality (\ref{f1}). Using a density argument, we may assume $u\in C^{\infty}\left(\mathbb{S}^{n-1}\right)$. Recalling the identities
(\ref{nablagdsc}) and (\ref{deltagdsc})
in Lemma \ref{gradlap}, we can compute
\begin{equation}\label{deltsn}
\Delta_{\mathbb{S}^{n-1}}
\sin{d}=-\sin{d}
+(n-2)\frac{\cos^{2}{d}}{\sin{d}}.
\end{equation}
Integrating both sides of (\ref{deltsn})
against $|u|^{p}/\sin^{p-1}{d}$ over $\mathbb{S}^{n-1}$, then employing the divergence theorem, we obtain
\begin{align}
\nonumber
&(n-2)\int_{\mathbb{S}^{n-1}}
\frac{|u|^p}{\sin^{p}{d}}\cos^{2}{d}\,
d\sigma_{n-1}-\int_{\mathbb{S}^{n-1}}
\frac{|u|^p}{\sin^{p-2}{d}}
d\sigma_{n-1}
\\ \nonumber&=
\int_{\mathbb{S}^{n-1}}\frac{|u|^p}{\sin^{p-1}{d}}
\Delta_{\mathbb{S}^{n-1}}
\sin{d}\,d\sigma_{n-1}\\
\nonumber
&=-\int_{\mathbb{S}^{n-1}}
\nabla_{\mathbb{S}^{n-1}} \left(\frac{|u|^p}{\sin^{p-1}{d}}\right)\cdot
\nabla_{\mathbb{S}^{n-1}} \sin{d}\,
d\sigma_{n-1}\\
&=
\nonumber
\int_{\mathbb{S}^{n-1}}
\frac{-p|u|^{p-2}u \nabla_{\mathbb{S}^{n-1}} u
\cdot \nabla_{\mathbb{S}^{n-1}}\sin{d}}
{\sin^{p-1}{d}}
d\sigma_{n-1}+\\&
\label{prtsf1}
\;\;\;+(p-1)\int_{\mathbb{S}^{n-1}}
\frac{|u|^p}{\sin^{p}{d}}\cos^{2}{d}\,
d\sigma_{n-1}.
\end{align}
Observe that we simplified the latter integral using the fact $|\nabla_{\mathbb{S}^{n-1}}d|=1$. So far, it suffices to require that $p>1$ to make sense of
the gradient of $|u|^{p}$. Invoking H\"{o}lder's inequality then applying Young's inequality and using (\ref{nablagdsc}) once more, we can bound
\begin{align}
\nonumber
&\int_{\mathbb{S}^{n-1}}
\frac{-p|u|^{p-2}u \nabla_{\mathbb{S}^{n-1}} u
\cdot \nabla_{\mathbb{S}^{n-1}} \sin{d}}
{\sin^{p-1}{d}}
d\sigma_{n-1}\\
&\nonumber \leq
p\left(\int_{\mathbb{S}^{n-1}}
\frac{|u|^{p}}
{|\sin{d}|^{p}}
d\sigma_{n-1}\right)^{\frac{p-1}{p}}
\left(
\int_{\mathbb{S}^{n-1}}
|\nabla_{\mathbb{S}^{n-1}} u\cdot \nabla_{\mathbb{S}^{n-1}} \sin{d}|^{p}
d\sigma_{n-1}\right)^{\frac{1}{p}}
\\
&\label{yngf1f1} \leq
(p-1)\beta^{\frac{p}{p-1}}\int_{\mathbb{S}^{n-1}}
\frac{|u|^{p}}
{|\sin{d}|^{p}}
d\sigma_{n-1}+
\frac{1}{\beta^{p}}
\int_{\mathbb{S}^{n-1}}
|\nabla_{\mathbb{S}^{n-1}} u|^{p}|\cos{d}|^{p}
d\sigma_{n-1},
\end{align}
with $\beta>0$ as yet undetermined. Plugging the estimate (\ref{yngf1f1}) into the inequality (\ref{prtsf1}) then rearranging gives
\begin{align}
\nonumber
&(n-p-1)\int_{\mathbb{S}^{n-1}}
\frac{|u|^p}{\sin^{p}{d}}\cos^{2}{d}\,
d\sigma_{n-1}-(p-1)\beta^{\frac{p}{p-1}}
\int_{\mathbb{S}^{n-1}}
\frac{|u|^p}{\sin^{p}{d}}d\sigma_{n-1}\\
&\label{prtsf1fn}
\leq\frac{1}{\beta^{p}}
\int_{\mathbb{S}^{n-1}}
|\nabla_{\mathbb{S}^{n-1}} u|^{p}|\cos{d}|^{p}
d\sigma_{n-1}+
\int_{\mathbb{S}^{n-1}}\frac{|u|^p}{\sin^{p-2}{d}}
d\sigma_{n-1}.
\end{align}
Note here that Remark \ref{rmcon} justifies this manipulation of the terms of (\ref{prtsf1}).
We proceed from (\ref{prtsf1fn}) by simply replacing the factor $\cos^{2}{d}$
by $1-\sin^{2}{d}$ in the first integral of
to get
\begin{align}
\nonumber
&\left(n-p-1-(p-1)\beta^{\frac{p}{p-1}}\right)\beta^{p}
\int_{\mathbb{S}^{n-1}}
\frac{|u|^p}{|\sin{d}|^{p}}
d\sigma_{n-1}
\\
&\label{prtsf1fnfn2}
\leq
\int_{\mathbb{S}^{n-1}}
|\nabla_{\mathbb{S}^{n-1}} u|^{p}|\cos{d}|^{p}
d\sigma_{n-1}+\beta^{p}(n-p)
\int_{\mathbb{S}^{n-1}}\frac{|u|^p}{\sin^{p-2}{d}}
d\sigma_{n-1}.
\end{align}
The optimal value of $\beta$ for (\ref{prtsf1fnfn2}) is easily determined through finding the maximum point $t_{*}$ of the function
$t\mapsto t\left(n-p-1-(p-1)t^{\frac{1}{p-1}}\right)$
on $[0,+\infty[$. We find $t_{*}=\left(\frac{n-p-1}{p}\right)^{\frac{p-1}{p}}$.
Hence the inequality (\ref{prtsf1fnfn2}) takes the form
(\ref{f1}).\\
\indent With the exception of some technical details, the proof of (\ref{f3}) is similar to that of (\ref{f1}). Instead of using (\ref{deltsn}), we capitalize on (\ref{deltagdsc}). Let $2\leq p<n-1$. Integration by parts on $\mathbb{S}^{n-1}$ yields
\begin{eqnarray}
\nonumber \int_{\mathbb{S}^{n-1}}\frac{|u|^{p}}{|\tan{d}|^{p}}
d\sigma_{n-1}&=&\frac{1}{n-2}
\int_{\mathbb{S}^{n-1}}
\frac{|u|^{p}}{|\tan{d}|^{p-2}\tan{d}}
\Delta_{\mathbb{S}^{n-1}} d\, d\sigma_{n-1}\\
\nonumber
&=&
\frac{-1}{n-2}
\int_{\mathbb{S}^{n-1}}\nabla_{\mathbb{S}^{n-1}} d\cdot
\nabla_{\mathbb{S}^{n-1}} \left(
\frac{|u|^{p}}{|\tan{d}|^{p-2}\tan{d}}\right)
d\sigma_{n-1}\\
&=&
\nonumber\frac{1}{n-2}
\int_{\mathbb{S}^{n-1}}
\frac{-p|u|^{p-2}u \nabla_{\mathbb{S}^{n-1}} u
\cdot \nabla_{\mathbb{S}^{n-1}} d}
{|\tan{d}|^{p-2}\tan{d}}
d\sigma_{n-1}+\\
&&\label{prts}+\frac{p-1}{n-2}\int_{\mathbb{S}^{n-1}}
\frac{|u|^{p}}
{|\tan{d}|^{p}}\frac{1}{\cos^{2}{d}}
d\sigma_{n-1}.
\end{eqnarray}
Observe that the restriction $2\leq p<n-1$
is necessary to make sense of $\nabla_{\mathbb{S}^{n-1}}|\cos{d}|^{p-2}\cos{d}$.
It also guarantees the convergence of the integral $\displaystyle
\int_{\mathbb{S}^{n-1}}
\frac{1}{|\tan{d}|^{p}}\frac{1}{\cos^{2}{d}}
d\sigma_{n-1}$. This is inferred by formula (\ref{frml}) that asserts
\begin{equation*}
\int_{\mathbb{S}^{n-1}}
\frac{d\sigma_{n-1}}{\sin^{p}{d}\cos^{2-p}{d}}
=2C_{n}\int_{0}^{1}
\frac{ds}{s^{2-p}(1-s^2)^{-\frac{n-p+3}{2}}}.
\end{equation*}
Since $|\nabla_{\mathbb{S}^{n-1}} d|=1$, then,
applying
H\"{o}lder's inequality followed by
Young's inequality analogously to (\ref{yngf1f1})  gives
\begin{align}
\nonumber
&\int_{\mathbb{S}^{n-1}}
\frac{-p|u|^{p-2}u \nabla_{\mathbb{S}^{n-1}} u
\cdot \nabla_{\mathbb{S}^{n-1}} d}
{|\tan{d}|^{p-2}\tan{d}}
d\sigma_{n-1}\\
\label{yng}
&\leq
(p-1)\beta^{\frac{p}{p-1}}\int_{\mathbb{S}^{n-1}}
\frac{|u|^{p}}
{|\tan{d}|^{p}}
d\sigma_{n-1}+
\beta^{-p}
\int_{\mathbb{S}^{n-1}}
|\nabla_{\mathbb{S}^{n-1}} u|^{p}
d\sigma_{n-1},
\end{align}
for any $\beta>0$. We can also split
\begin{equation}\label{tncs}
\int_{\mathbb{S}^{n-1}}
\frac{|u|^{p}}
{|\tan{d}|^{p}}\frac{1}{\cos^{2}{d}}d\sigma_{n-1}
=\int_{\mathbb{S}^{n-1}}
\frac{|u|^{p}d\sigma_{n-1}}
{|\tan{d}|^{p-2}}
+\int_{\mathbb{S}^{n-1}}
\frac{|u|^{p}d\sigma_{n-1}}
{|\tan{d}|^{p}}.
\end{equation}
Returning to (\ref{prts}) with
(\ref{yng}) and (\ref{tncs}) we deduce that
\begin{align}
\nonumber &\beta^{p}\left(n-p-1-(p-1)\beta^{\frac{p}{p-1}}\right)
\int_{\mathbb{S}^{n-1}}
\frac{|u|^{p}}{|\tan{d}|^{p}}d\sigma_{n-1}\\
\label{f0}
&\leq \int_{\mathbb{S}^{n-1}}
|\nabla_{\mathbb{S}^{n-1}} u|^{p}
d\sigma_{n-1}+\beta^{p}(p-1)\int_{\mathbb{S}^{n-1}}
\frac{|u|^{p}}
{|\tan{d}|^{p-2}}d\sigma_{n-1}.
\end{align}
The optimal value of $\beta$ for (\ref{f0}) is $\alpha^{\frac{p-1}{p}}_{n,p}$. This proves the inequality (\ref{f3}).
\end{proof}
\begin{thm}(Sharpness of the inequalities (\ref{f1})-(\ref{f3}))\\
The constants on both sides of the inequality (\ref{f1}) are sharp. Precisely, we have
\begin{align}
\label{f1shrp1}&\sup_{u\in W^{1,p}\left(\mathbb{S}^{n-1}\right)\setminus\left\{ 0\right\}}
{\frac{\alpha^{p}_{n,p}S_{p}(u)}{G_{p}(u) +
(n-p)\alpha^{p-1}_{n,p}\widetilde{S}_{p}(u)}}=1,\;
1<p<n-1\\
\label{f1shrp2}&\sup_{u\in W^{1,p}\left(\mathbb{S}^{n-1}\right)\setminus\left\{ 0\right\}}
{\frac{\alpha_{n,p}S_{p}(u)-(n-p)
\widetilde{S}_{p}(u)}{G_{p}(u)}}=\alpha^{1-p}_{n,p},\;
1<p<n-1\\
\label{f1shrp3}&\sup_{u\in W^{1,p}\left(\mathbb{S}^{n-1}\right)\setminus\left\{ 0\right\}}
{\frac{\alpha^{p}_{n,p}S_{p}(u)-G_{p}(u)}{
\widetilde{S}_{p}(u)}}=(n-p)\alpha^{p-1}_{n,p},\;
1<p<n/2.
\end{align}
All the constants involved in the inequality (\ref{f3}) are sharp for all $\,2\leq p<n-1$. Precisely
\begin{eqnarray}
\label{f3shrp1}\sup_{u\in W^{1,p}\left(\mathbb{S}^{n-1}\right)\setminus\left\{ 0\right\}}
{\frac{\alpha^{p}_{n,p}T_{p}(u)}{F_{p}(u) +
(p-1)\alpha^{p-1}_{n,p}\widetilde{T}_{p}(u)}}&=&1,\\
\label{f3shrp2}\sup_{u\in W^{1,p}\left(\mathbb{S}^{n-1}\right)\setminus\left\{ 0\right\}}
{\frac{\alpha_{n,p}T_{p}(u)-(p-1)
\widetilde{T}_{p}(u)}{F_{p}(u)}}&=&\alpha^{1-p}_{n,p},\\
\label{f3shrp3}\sup_{u\in W^{1,p}\left(\mathbb{S}^{n-1}\right)\setminus\left\{ 0\right\}}
{\frac{\alpha^{p}_{n,p}T_{p}(u)-F_{p}(u)}{
\widetilde{T}_{p}(u)}}&=&(p-1)\alpha^{p-1}_{n,p}.
\end{eqnarray}
\end{thm}
\begin{rem}
The values $\frac{n}{2}\leq p< n-1$ can be admitted in (\ref{f1shrp3}) if the supremum
is taken over nontrivial functions
in $L^{p}(\mathbb{S}^{n-1})$ with weak gradient
in the weighted space
$L^{p}(\mathbb{S}^{n-1};|
\cos{d(\Theta_{n-1},\Phi_{n-1})}|^{p}d\sigma_{n-1})$.
\end{rem}
\begin{proof}
Fix $n\geq 4$ and $1< p<n-1$. Consider the function
\begin{equation*}
u_{\epsilon}(\Theta_{n-1}):=
\left|\cot{d(\Theta_{n-1},\Phi_{n-1})}\right|^{
{(n-p-1-2\epsilon)}/{p}}\cos{d(\Theta_{n-1},\Phi_{n-1})}
\end{equation*}
on $\mathbb{S}^{n-1}\setminus
\left\{\pm \Phi_{n-1}\right\}$. Verifiably $u_{\epsilon}\in W^{1,p}\left(\mathbb{S}^{n-1}\right)$ for every $\epsilon>0$. Moreover
\begin{equation}\label{intusnp}
\int_{\mathbb{S}^{n-1}}
\frac{|u_{\epsilon}|^{p}}{\sin^{p}{d}}d\sigma_{n-1}
=I_{n,p}(\epsilon),
\end{equation}
where
\begin{equation*}
I_{n,p}(\epsilon):=\int_{\mathbb{S}^{n-1}}
\frac{\left|\cos{d}\right|^{n-1-2\epsilon}}
{\left(\sin{d}\right)^{n-1-2\epsilon}}
d\sigma_{n-1}.
\end{equation*}
Exploiting formula (\ref{frml}) we find
\begin{eqnarray}
\nonumber
I_{n,p}(\epsilon)&=&
C_{n}\int_{-1}^{1}\frac{
|t|^{n-1-2\epsilon}\left(\sqrt{1-t^2}\right)^{n-3}}{
\left(\sqrt{1-t^2}\right)^{n-1-2\epsilon}}dt\\
\label{inpepslon}&=&
2C_{n}\int_{0}^{1}t^{n-1-2\epsilon}
\left(1-t^2\right)^{-1+\epsilon}dt\,=\,C_{n}
\frac{\Gamma{(\epsilon)}\Gamma{\left(
\frac{n}{2}-\epsilon\right)}}{\Gamma{\left(
\frac{n}{2}\right)}}.
\end{eqnarray}
Notice that $I_{n,p}(\epsilon)$ is finite for every $\epsilon>0$ but blows up in the limit. Additionally
\begin{align}
\nonumber
&\int_{\mathbb{S}^{n-1}}
\frac{|u_{\epsilon}|^{p}}{\sin^{p-2}{d}}
d\sigma_{n-1}=
\int_{\mathbb{S}^{n-1}}
\frac{|\cos{d}|^{n-1-2\epsilon}}{
(\sin{d})^{n-3-2\epsilon}}
d\sigma_{n-1}\\
\label{u2nsmn2}
&=2C_{n}\int_{0}^{1}t^{n-1-2\epsilon}
\left(1-t^2\right)^{\epsilon}dt
=C_{n}
\frac{\epsilon\Gamma{(\epsilon)}\Gamma{\left(
\frac{n}{2}-\epsilon\right)}}{\frac{n}{2}\Gamma{\left(
\frac{n}{2}\right)}}.
\end{align}
Furthermore, for all $\Theta_{n-1}\notin \left\{\pm\Phi_{n-1}\right\}\cup
\left\{\Theta_{n-1}\in{\mathbb{S}}^{n-1}:
d\left(\Theta_{n-1},\Phi_{n-1}\right)={\pi}/{2}
\right\}$
\begin{equation*}
\begin{split}
\cos{d}\,\nabla_{\mathbb{S}^{n-1}} u_{\epsilon}
=&-
\frac{n-p-1-2\epsilon}{p}\left|\cot{d}\right|^{
\frac{n-1-2\epsilon}{p}}
\sign{(\cos{d})}\,\nabla_{\mathbb{S}^{n-1}}d +\\
&-|\cot{d}|^{
\frac{n-p-1-2\epsilon}{p}}\sin{d}\cos{d}\,
\nabla_{\mathbb{S}^{n-1}}d.
\end{split}
\end{equation*}
And Minkowski's inequality implies
\begin{equation*}
\begin{split}
\left\|\cos{d}\,\nabla_{\mathbb{S}^{n-1}} u_{\epsilon} \right\|^{p}_{L^{p}\left(\mathbb{S}^{n-1}\right)}\leq&\;
\biggl(\left(\frac{n-p-1-2\epsilon}{p}\right)
\left\|\left|\cot{d}\right|^{
\frac{n-1-2\epsilon}{p}}\right\|_{L^{p}
\left(\mathbb{S}^{n-1}\right)}+\\
&+\left\||\cot{d}|^{\frac{n-p-1-2\epsilon}{p}}
\sin{d}\cos{d}\right\|
_{L^{p}\left(\mathbb{S}^{n-1}\right)}\biggr)^{p}.
\end{split}
\end{equation*}
with $0<\epsilon<(n-p-1)/2$. That is
\begin{equation*}
\begin{split}
&\int_{\mathbb{S}^{n-1}}
|\nabla_{\mathbb{S}^{n-1}} u_{\epsilon}|^{p}\left|\cos{d}\right|^{p}
d\sigma_{n-1}
\leq\left(\alpha_{n,p}-\frac{2}{p}\epsilon\right)^{p}
I_{n,p}(\epsilon)\times\\&
\left(1+\left(
{\int_{\mathbb{S}^{n-1}}
\frac{\left|\cos{d}\right|^{n-1-2\epsilon}}
{\left(\sin{d}\right)^{n-2p-1-2\epsilon}}
d\sigma_{n-1}}\right)^{{1}/{p}}\bigg{/}
\left({\left(\alpha_{n,p}-\frac{2}{p}\epsilon\right)^{p}
I_{n,p}(\epsilon)}\right)^{{1}/{p}}\right)^{p}.
\end{split}
\end{equation*}
Since, by (\ref{frml}), we have
\begin{equation*}
\int_{\mathbb{S}^{n-1}}
\frac{\left|\cos{d}\right|^{n-1-2\epsilon}}
{\left(\sin{d}\right)^{n-2p-1-2\epsilon}}
d\sigma_{n-1}=C_{n}
\frac{\Gamma \left(\frac{n}{2}-\epsilon\right) \Gamma (p+\epsilon)}{ \Gamma \left(\frac{n}{2}+p\right)},
\end{equation*}
then, we can simplify
\begin{equation}
\label{gradcos}\int_{\mathbb{S}^{n-1}}
|\nabla_{\mathbb{S}^{n-1}} u_{\epsilon}|^{p}\left|\cos{d}\right|^{p}
d\sigma_{n-1}
\leq\left(\alpha_{n,p}-\frac{2}{p}\epsilon\right)^{p}
I_{n,p}(\epsilon)
\left(1+\Lambda_{n,p}(\epsilon)\right)^{p},
\end{equation}
where
\begin{equation*}
\Lambda_{n,p}(\epsilon):=\left(
\frac{\Gamma{(p+\epsilon)}
\Gamma{\left(\frac{n}{2}\right)}}{
\Gamma{\left(\frac{n}{2}+p\right)}
\left(\alpha_{n,p}-\frac{2}{p}\epsilon\right)^{p}
}\frac{1}{\Gamma{(\epsilon)}}
\right)^{{1}/{p}}.
\end{equation*}
Consequently, putting together the inequality (\ref{f1}), (\ref{intusnp}), (\ref{inpepslon}), (\ref{u2nsmn2}), and the estimate (\ref{gradcos}) implies
\begin{equation*}
\begin{split}
&1\geq \frac{\alpha^{p}_{n,p}
S_{p}(u_{\epsilon})}{G_{p}(u_{\epsilon}) +(n-p)\alpha^{p-1}_{n,p}\widetilde{S}_{p}(u_{\epsilon})}\\
&\geq
\frac{\alpha^{p}_{n,p}}{\left(\alpha_{n,p}-\frac{2}{p}
\epsilon\right)^{p}
\left(1+\Lambda_{n,p}(\epsilon)\right)^{p} +(n-p)\alpha^{p-1}_{n,p}
\frac{2}{n}\epsilon}\longrightarrow 1
\end{split}
\end{equation*}
as $\epsilon\longrightarrow 0^{+}$
by the continuity of the gamma function
on $]0,\infty[$, and the fact
that $\lim_{\epsilon\rightarrow 0^{+}}\Gamma{(\epsilon)}=+\infty$ that makes
 $\lim_{\epsilon\rightarrow 0^{+}}\Lambda_{n,p}(\epsilon)=0$. This
squeeze along with the inequality (\ref{f1}) proves (\ref{f1shrp1}).
In the same fashion
\begin{equation*}
1\geq\;\frac{\alpha^{p}_{n,p}
S_{p}(u_{\epsilon})-
(n-p)\alpha^{p-1}_{n,p}\widetilde{S}_{p}(u_{\epsilon})}{
G_{p}(u_{\epsilon})}\;\geq\;
\frac{\alpha^{p}_{n,p}-
(n-p)\alpha^{p-1}_{n,p}
\frac{2}{n}\epsilon}{\left(\alpha_{n,p}-\frac{2}{p}
\epsilon\right)^{p}
\left(1+\Lambda_{n,p}(\epsilon)\right)^{p} }\longrightarrow 1
\end{equation*}
when $\epsilon\longrightarrow 0^{+}$.
This shows (\ref{f1shrp2}) in the light of (\ref{f1}).
We proceed to prove (\ref{f1shrp3}) that shows that the constant $n-p$ on the right hand-side of (\ref{f1})
is the smallest possible for all $1<p<\frac{n}{2}$.
Define the function
\begin{equation*}
\tilde{u}_{\epsilon}(\Theta_{n-1}):=
\left|\cot{d(\Theta_{n-1},\Phi_{n-1})}\right|^{
{(n-p-1-2\epsilon)}/{p}},\;
\Theta_{n-1} \in \mathbb{S}^{n-1}\setminus
\left\{\pm \Phi_{n-1}\right\}.
\end{equation*}
For $\epsilon>0$, we have
$\tilde{u}_{\epsilon}\in L^{p}(\mathbb{S}^{n-1})$,
$1<p<n-1$ and $\nabla_{\mathbb{S}^{n-1}}
\tilde{u}_{\epsilon}\in L^{p}(\mathbb{S}^{n-1})$,
$1<p<n/2$ as shown by the following calculations:
\begin{equation}\label{n21}
\begin{split}
&\int_{\mathbb{S}^{n-1}}
\frac{|\tilde{u}_{\epsilon}|^{p}}{\sin^{p}{d}}d\sigma_{n-1}
=
\int_{\mathbb{S}^{n-1}}
\frac{\left|\cos{d}\right|^{n-p-1-2\epsilon}}
{\left(\sin{d}\right)^{n-1-2\epsilon}}
d\sigma_{n-1}\\
&=C_{n}\int_{-1}^{1}
t^{n-p-1-2\epsilon}(1-t^2)^{-1+\epsilon}dt=
C_{n}\frac{\Gamma{(\epsilon)} \Gamma{ \left(\frac{n-p}{2}-\epsilon\right)}}{ \Gamma{ \left(\frac{n-p}{2}\right)}},
\end{split}
\end{equation}
\begin{equation}\label{n22}
\begin{split}
&\int_{\mathbb{S}^{n-1}}
\frac{|\tilde{u}_{\epsilon}|^{p}}{
\sin^{p-2}{d}}d\sigma_{n-1}=
\int_{\mathbb{S}^{n-1}}
\frac{\left|\cos{d}\right|^{n-p-1-2\epsilon}}
{\left(\sin{d}\right)^{n-3-2\epsilon}}
d\sigma_{n-1}\\
&=C_{n}\int_{-1}^{1}
t^{n-p-1-2\epsilon}(1-t^2)^{\epsilon}dt=
C_{n}\frac{\epsilon\Gamma{(\epsilon)} \Gamma{ \left(\frac{n-p}{2}-\epsilon\right)}}{\frac{n-p}{2} \Gamma{ \left(\frac{n-p}{2}\right)}},
\end{split}
\end{equation}
\begin{equation}\label{n23}
\begin{split}
&\int_{\mathbb{S}^{n-1}}
\left|\nabla_{\mathbb{S}^{n-1}}\tilde{u}_{\epsilon}
\right|^{p}
\left|\cos{d}\right|^{p}
d\sigma_{n-1}=\left(\alpha_{n,p}-\frac{2}{p}
\epsilon\right)^{p}
\int_{\mathbb{S}^{n-1}}
\frac{\left|\cos{d}\right|^{n-p-1-2\epsilon}}
{\left(\sin{d}\right)^{n-1-2\epsilon}}
d\sigma_{n-1}\\
&=C_{n}
\left(\alpha_{n,p}-\frac{2}{p}\epsilon\right)^{p}
\frac{\Gamma{(\epsilon)} \Gamma{ \left(\frac{n-p}{2}-\epsilon\right)}}{ \Gamma{ \left(\frac{n-p}{2}\right)}}.
\end{split}
\end{equation}
Combining (\ref{n21}) - (\ref{n23}) yields
\begin{equation*}
\frac{\alpha^{p}_{n,p}
S_{p}(\tilde{u}_{\epsilon})-G_{p}(\tilde{u}_{\epsilon})}
{\alpha^{p-1}_{n,p}\widetilde{S}_{p}(\tilde{u}_{\epsilon})}
=\frac{n-p}{2}\frac{
\alpha^{p}_{n,p}-
(\alpha_{n,p}-\frac{2}{p}\epsilon)^{p}
}{\epsilon}
\longrightarrow n-p
\end{equation*}
as $\epsilon\longrightarrow 0^{+}$. This convergence
together with the inequality
(\ref{f3}) confirm  (\ref{f3shrp3}).\\
\indent Now we turn our attention to (\ref{f3shrp1})-(\ref{f3shrp3}).
Let $2\leq p<n-1$, $n\geq 4$, and define the function
\begin{equation*}
v_{\epsilon}(\Theta_{n-1}):=(\sin{d(\Theta_{n-1},\Phi_{n-1})})^{-
{(n-p-1-2\epsilon)}/{p}}
\end{equation*}
on $\mathbb{S}^{n-1}\setminus
\left\{\pm \Phi_{n-1}\right\}$. Evidently
$v_{\epsilon}\in W^{1,p}\left(\mathbb{S}^{n-1}\right)$ for every $\epsilon>0$. Furthermore, $0<\epsilon<(n-p-1)/2$, we have
\begin{equation*}
\int_{\mathbb{S}^{n-1}}
\frac{|v_{\epsilon}|^{p}}{|\tan{d}|^{p}}d\sigma_{n-1}
=J_{n,p}(\epsilon),\;
\int_{\mathbb{S}^{n-1}}
|\nabla_{\mathbb{S}^{n-1}} v_{\epsilon}|^{p}
d\sigma_{n-1}=
\left(\alpha_{n,p}-\frac{2}{p}\epsilon\right)^{p}
J_{n,p}(\epsilon),
\end{equation*}
where
\begin{equation*}
J_{n,p}(\epsilon):=\int_{\mathbb{S}^{n-1}}
\frac{|\cos{d}|^{p}}{(\sin{d})^{n-1-2\epsilon}}
d\sigma_{n-1}=
2C_{n}\int_{0}^{1}t^{p}
\left(1-t^2\right)^{-1+\epsilon}dt=C_{n}
\frac{\Gamma{(\epsilon)}\Gamma{\left(
\frac{p+1}{2}\right)}}{\Gamma{\left(
\frac{p+1}{2}+\epsilon\right)}},
\end{equation*}
using formula (\ref{frml}). We also have
\begin{equation*}
\begin{split}
&\int_{\mathbb{S}^{n-1}}
\frac{|v_{\epsilon}|^{p}}{|\tan{d}|^{p-2}}
d\sigma_{n-1}=
\int_{\mathbb{S}^{n-1}}
\frac{|\cos{d}|^{p-2}}{(\sin{d})^{n-3-2\epsilon}}
d\sigma_{n-1}\\
&=2C_{n}\int_{0}^{1}t^{p-2}
\left(1-t^2\right)^{\epsilon}dt=C_{n}
\frac{\epsilon\Gamma{(\epsilon)}\Gamma{\left(
\frac{p-1}{2}\right)}}{\Gamma{\left(
\frac{p+1}{2}+\epsilon\right)}}.
\end{split}
\end{equation*}
Consequently
\begin{equation*}
\begin{split}
&\frac{\alpha^{p}_{n,p}
T_{p}(v_{\epsilon})}{F_{p}(v_{\epsilon}) +(p-1)\alpha^{p-1}_{n,p}\widetilde{T}_{p}(v_{\epsilon})}\\
&=\frac{\alpha^{p}_{n,p}
\Gamma{\left(
\frac{p+1}{2}\right)}}{
(\alpha_{n,p}-\frac{2}{p}\epsilon)^{p}\Gamma{\left(
\frac{p+1}{2}\right)}
 +
(p-1)\epsilon\alpha^{p-1}_{n,p}\Gamma{\left(
\frac{p-1}{2}\right)}}\longrightarrow 1
\end{split}
\end{equation*}
as $\epsilon\longrightarrow 0^{+}$. This together with the inequality (\ref{f3}) proves (\ref{f3shrp1}).
In addition
\begin{equation*}
\begin{split}
&\frac{\alpha^{p}_{n,p}
T_{p}(v_{\epsilon})-
(p-1)\alpha^{p-1}_{n,p}\widetilde{T}_{p}(v_{\epsilon})}{
F_{p}(v_{\epsilon})}\\
&=\frac{\alpha^{p}_{n,p}
\Gamma{\left(
\frac{p+1}{2}\right)}-
(p-1)\epsilon\alpha^{p-1}_{n,p}\Gamma{\left(
\frac{p-1}{2}\right)}}{
(\alpha_{n,p}-\frac{2}{p}\epsilon)^{p}\Gamma{\left(
\frac{p+1}{2}\right)}}\longrightarrow 1
\end{split}
\end{equation*}
when $\epsilon\longrightarrow 0^{+}$ which proves (\ref{f3shrp2}).
Finally, we show (\ref{f3shrp3}) that demonstrates that the constant
$p-1$ on the right hand-side of (\ref{f3})
is optimal. We  have
\begin{equation*}
\begin{split}
&\frac{\alpha^{p}_{n,p}
T_{p}(v_{\epsilon})-F_{p}(v_{\epsilon})}
{\alpha^{p-1}_{n,p}\widetilde{T}_{p}(v_{\epsilon})}
=\frac{\alpha^{p}_{n,p}
\Gamma{\left(
\frac{p+1}{2}\right)}-
(\alpha_{n,p}-\frac{2}{p}\epsilon)^{p}\Gamma{\left(
\frac{p+1}{2}\right)}}{\epsilon\alpha^{p-1}_{n,p}\Gamma{\left(
\frac{p-1}{2}\right)}}\\
&=\frac{
\left(\alpha^{p}_{n,p}-
(\alpha_{n,p}-\frac{2}{p}\epsilon)^{p}
\right)\frac{p-1}{2}\Gamma{\left(
\frac{p-1}{2}\right)}}{\epsilon\alpha^{p-1}_{n,p}\Gamma{\left(
\frac{p-1}{2}\right)}}
=\frac{
\left(\alpha^{p}_{n,p}-
(\alpha_{n,p}-\frac{2}{p}\epsilon)^{p}
\right)}{\epsilon}\frac{p-1}{2\alpha^{p-1}_{n,p}}
\longrightarrow p-1
\end{split}
\end{equation*}
using L'H\^{o}pital's rule.
\end{proof}
An important consequence of (\ref{f1shrp3})
is that, in any dimension $n\geq 4$, and for every $1<p<n-1$,  we can find $u\in C^{\infty}\left(\mathbb{S}^{n-1}\right)$ such that
\begin{eqnarray*}
\alpha^{p}_{n,p}S_{p}(u)&>& G_{p}(u) +
\alpha^{p-1}_{n,p}\widetilde{S}_{p}(u).
\end{eqnarray*}
It similarly follows from (\ref{f3shrp3}) that the inequality
\begin{eqnarray*}
\alpha^{p}_{n,p}T_{p}(u)&\leq& F_{p}(u) +
\alpha^{p-1}_{n,p}\widetilde{T}_{p}(u),\;p> 2
\end{eqnarray*}
does not hold true on $C^{\infty}\left(\mathbb{S}^{n-1}\right)$, $n\geq 4$.
More interestingly,
\begin{thm}
The inequality
\begin{eqnarray}\label{inqfls}
\alpha^{p}_{n,p}S_{p}(u)&\leq& F_{p}(u) +
\alpha^{p-1}_{n,p}\widetilde{S}_{p}(u)
\end{eqnarray}
is generally false on $W^{1,p}\left(\mathbb{S}^{n-1}\right)$ for every $1<p<n-1$, $n\geq 4$.
In particular, there exists  $u\in H^{1}\left(\mathbb{S}^{n-1}\right)$ such that
\begin{equation*}
\left(\frac{n-3}{2}\right)^{2}\int_{\mathbb{S}^{n-1}} \frac{u^{2}}{\sin^{2}{d}}d\sigma_{n-1}>
\int_{\mathbb{S}^{n-1}} |\nabla_{\mathbb{S}^{n-1}}u|^{2}d\sigma_{n-1} +
\frac{n-3}{2}\int_{\mathbb{S}^{n-1}} {u}^{2}d\sigma_{n-1}
\end{equation*}
for every $n>4$.
\end{thm}
\begin{proof}
If we test the inequality (\ref{inqfls}) with a constant function we find it false for $1<p<n/2$.
We provide an explicit counterexample in $W^{1,p}\left(\mathbb{S}^{n-1}\right)$
for which (\ref{inqfls}) fails for all $1<p<n-1$.
Define
\begin{equation*}
w_{\epsilon}(\Theta_{n-1}):=
\left(\frac{1+\cos{d(\Theta_{n-1},\Phi_{n-1})}}{\sin{d(\Theta_{n-1},\Phi_{n-1})}}\right)^{
\frac{n-p-1-2\epsilon}{p}},\quad 1<p<n-1.
\end{equation*}
When  $\epsilon$ is sufficiently small, we have
\begin{eqnarray}
\nonumber
L_{n,p}(\epsilon):&=&\left(\frac{n-p-1}{p}\right)^{p}
\int_{\mathbb{S}^{n-1}}
\frac{|w_{\epsilon}|^{p}}{|\sin{d}|^{p}}d\sigma_{n-1}-
\int_{\mathbb{S}^{n-1}}
|\nabla_{\mathbb{S}^{n-1}} w_{\epsilon}|^{p}
d\sigma_{n-1}\\
\label{lpne}&=&
\left(\left(\frac{n-p-1}{p}\right)^{p}-
\left(\frac{n-p-1}{p}-\frac{2}{p}\epsilon\right)^{p}\right)
K_{n,p}(\epsilon)
\end{eqnarray}
with
\begin{equation*}
K_{n,p}(\epsilon):=\int_{\mathbb{S}^{n-1}}
\frac{\left(1+\cos{d}\right)^{n-p-1-2\epsilon}}
{(\sin{d})^{n-1-2\epsilon}}
d\sigma_{n-1}.
\end{equation*}
Formula (\ref{frml}) helps us determine
the exact value of $K_{n,p}(\epsilon)$. Indeed, upon translating $t\rightarrow t-1$ then rescaling
$t\rightarrow 2t$, we discover that
\begin{eqnarray}
\nonumber
K_{n,p}(\epsilon)&=&
C_{n}\int_{-1}^{1}
(1+t)^{n-p-1-2\epsilon}\left(1-t^{2}\right)^{-1+
\epsilon}dt\\
\nonumber&=&2^{n-p-2}
C_{n}\int_{0}^{1}t^{n-p-2-\epsilon}
\left(1-t\right)^{-1+\epsilon}dt\\
\label{jv}
&=&2^{n-p-2}
C_{n}
\frac{\Gamma{(\epsilon)}\Gamma{\left(
n-p-1-\epsilon\right)}}{\Gamma{\left(
n-p-1\right)}}.
\end{eqnarray}
Therefore, substituting in (\ref{lpne}) from (\ref{jv}), then recognizing the limit $\lim_{\epsilon\rightarrow 0^{+}} \epsilon\Gamma{(\epsilon)}=1$, and the continuity of the gamma function on $]0,\infty[$, we obtain
\begin{eqnarray}
\nonumber
\lim_{\epsilon\rightarrow 0^{+}} L_{n,p}(\epsilon)&=&
2^{n-p-2}
C_{n}\lim_{\epsilon\rightarrow 0^{+}}
\frac{\left(\frac{n-p-1}{p}\right)^{p}-
\left(\frac{n-p-1}{p}-\frac{2}{p}\epsilon\right)^{p}}{
\epsilon}\lim_{\epsilon\rightarrow 0^{+}}
\epsilon\Gamma{(\epsilon)}\\
\label{lmlnpe}&=&
2^{n-p-1}
C_{n}\left(\frac{n-p-1}{p}\right)^{p-1}.
\end{eqnarray}
On the other hand
\begin{eqnarray}
\nonumber
\int_{\mathbb{S}^{n-1}}
\frac{|w_{\epsilon}|^{p}}{|\sin{d}|^{p-2}}d\sigma_{n-1}&=&
\int_{\mathbb{S}^{n-1}}
\frac{\left(1+\cos{d}\right)^{n-p-1-2\epsilon}}
{(\sin{d})^{n-3-2\epsilon}}
d\sigma_{n-1}\\
\nonumber&=&2^{n-p}
C_{n}\int_{0}^{1}t^{n-p-1-\epsilon}
\left(1-t\right)^{\epsilon}dt\\
\label{rnpeint}
&=&2^{n-p}
C_{n}
\frac{\Gamma{(1+\epsilon)}\Gamma{\left(
n-p-\epsilon\right)}}{(n-p)\Gamma{\left(
n-p\right)}}.
\end{eqnarray}
Hence, defining
\begin{eqnarray*}
R_{n,p}(\epsilon):&=&\left(\frac{n-p-1}{p}\right)^{p-1}
\int_{\mathbb{S}^{n-1}}
\frac{|w_{\epsilon}|^{p}}{|\sin{d}|^{p-2}}d\sigma_{n-1},
\end{eqnarray*}
we see from (\ref{rnpeint}) that
\begin{equation}
\label{lmrnpe}
\lim_{\epsilon\rightarrow 0^{+}} R_{n,p}(\epsilon)=
\frac{2^{n-p}C_{n}}{n-p}
\left(\frac{n-p-1}{p}\right)^{p-1}.
\end{equation}
Comparing (\ref{lmlnpe})  against (\ref{lmrnpe}) we conclude that, given $1<p<n-1$, there exists
$0<\epsilon_{0}<(n-p-1)/2$ such that
\begin{equation*}
\alpha^{p}_{n,p}S_{p}(w_{\epsilon})>F_{p}(w_{\epsilon})+
\alpha^{p-1}_{n,p}\widetilde{S}_{p}(w_{\epsilon})
\end{equation*}
for each $\epsilon<\epsilon_{0}$.
\end{proof}
\section{Critical $L^{p}$ Hardy inequalities}
Let $n\geq 2$ and define the following nonlinear positive functionals on $W^{1,n}(\mathbb{S}^{n}\longrightarrow \mathbb{R})$:
\begin{eqnarray*}
U_{n}(u) &:=& \int_{\mathbb{S}^{n}}
\frac{|u|^{n}}{\sin^{n}{d}\left(
\log{\frac{e}{\sin{d}}}\right)^{n}}d\sigma_{n}, \\
\widetilde{U}_{n}(u) &:=& \int_{\mathbb{S}^{n}}
\frac{|u|^{n}}{\sin^{n-2}{d}
\left(\log{\frac{e}{\sin{d}}}\right)^{n-1}}d\sigma_{n}, \\
V_{n}(u) &:=& \int_{\mathbb{S}^{n}}
\frac{|u|^{n}}{|\tan{d}|^{n}
\left(\log{\frac{e}{\sin{d}}}\right)^{n}}d\sigma_{n}, \\\widetilde{V}_{n}(u) &:=& \int_{\mathbb{S}^{n}}
\frac{|u|^{n}}{|\tan{d}|^{n-2}
\left(\log{\frac{e}{\sin{d}}}\right)^{n-1}}d\sigma_{n},\\
H_{n}(u) &:=& \int_{\mathbb{S}^{n}}
|\nabla_{\mathbb{S}^{n}} u|^{n}d\sigma_{n}, \\
Q_{n}(u) &:=& \int_{\mathbb{S}^{n}}
|\nabla_{\mathbb{S}^{n}} u|^{n}|\cos{d}|^{n} d\sigma_{n}.
\end{eqnarray*}
Define also the constant
\begin{equation*}
\gamma_{n}:=\frac{n-1}{n}.
\end{equation*}
\begin{thm}
Suppose $u\in W^{1,n}(\mathbb{S}^{n}\longrightarrow \mathbb{R})$ where $n\geq 2$. Then
$\frac{u}{\sin{d}\log{\left(\frac{e}{\sin{d}}\right)}},\,$
$\frac{u}{\tan{d}\log{\left(\frac{e}{\sin{d}}\right)}}\,$
are in $L^{n}(\mathbb{S}^{n})$. Furthermore
\begin{eqnarray}
\label{fc1} \gamma^{n}_{n}U_{n}(u)&\leq&Q_{n}(u) +
n\gamma^{n-1}_{n}\widetilde{U}_{n}(u),
\\
\label{fc2}
\gamma^{n}_{n}V_{n}(u)&\leq&H_{n}(u) +
(n-1)\gamma^{n-1}_{n}\widetilde{V}_{n}(u).
\end{eqnarray}
\end{thm}
\begin{rem}
Arguing as in Remark \ref{rmcon}, the functionals
$U_{n}$ and $V_{n}$ are bounded on continuous functions as the integrals $\displaystyle \int_{\mathbb{S}^{n}}
\frac{d\sigma_{n}}{\sin^{n}{d}\left(
\log{\frac{e}{\sin{d}}}\right)^{n}}$,
$\displaystyle\int_{\mathbb{S}^{n}}
\frac{d\sigma_{n}}{|\tan{d}|^{n}\left(
\log{\frac{e}{\sin{d}}}\right)^{n}}$ are convergent. This is clear thanks to formula (\ref{frml}) that ascertains
\begin{equation*}
\int_{\mathbb{S}^{n}}
\frac{d\sigma_{n}}{|\tan{d}|^{n}\left(
\log{\frac{e}{\sin{d}}}\right)^{n}}<\int_{\mathbb{S}^{n}}
\frac{d\sigma_{n}}{\sin^{n}{d}\left(
\log{\frac{e}{\sin{d}}}\right)^{n}}=C_{n}
\int_{-1}^{1}\frac{dt}{(1-t^2)
\left(\log{\frac{e}{\sqrt{1-t^2}}}\right)^{n}},
\end{equation*}
whereas the latter integral exists for all $n>1$.
\end{rem}
\begin{rem}
It is noteworthy that the integral
$\displaystyle
\int_{\mathbb{S}^{n}}
\frac{d\sigma_{n}}{|\tan{d}|^{n}\left|
\log{{c}{|\tan{d}|}}\right|^{m}}
$ is divergent for every $m\in \mathbb{R}$ and any $c>0$. Observe that
\begin{equation*}
\int_{\mathbb{S}^{n}}
\frac{d\sigma_{n}}{|\tan{d}|^{n}\left|
\log{{c}{|\tan{d}|}
}\right|^{m}}=2C_{n}\int_{0}^{1}\frac{s^n ds}{(1-s^2)
\left|\log{\left(\frac{c\sqrt{1-s^2}}{s}\right)}
\right|^{m}}.
\end{equation*}
\end{rem}
\begin{proof}
The proof of (\ref{fc1}) is
analogous to that of (\ref{f1}). Let $n\geq 2$ and use density to assume $u\in C^{\infty}(\mathbb{S}^{n})$.
Starting from (\ref{deltsn}), we integrate both sides against $\displaystyle
{|u|^n}/{\left(\sin^{n-1}{d}\log{\left(
{e}/{\sin{d}}\right)^{n-1}}\right)}$, then use the divergence theorem. We get
\begin{align}
\nonumber
&(n-1)\int_{\mathbb{S}^{n}}
\frac{|u|^n\cos^{2}{d}}{\sin^{n}{d}\left(
\log{\frac{e}{\sin{d}}}\right)^{n-1}}
d\sigma_{n}-\int_{\mathbb{S}^{n}}
\frac{|u|^n}{\sin^{n-2}{d}\left(
\log{\frac{e}{\sin{d}}}\right)^{n-1}}
d\sigma_{n}
\\ \nonumber&=
\int_{\mathbb{S}^{n}}\frac{|u|^n}{\sin^{n-1}{d}
\left(
\log{\frac{e}{\sin{d}}}\right)^{n-1}}
\Delta_{\mathbb{S}^{n}}
\sin{d}\,d\sigma_{n}\\
\nonumber
&=-\int_{\mathbb{S}^{n}}
\nabla_{\mathbb{S}^{n}} \left(\frac{|u|^n}{\sin^{n-1}{d}\left(
\log{\frac{e}{\sin{d}}}\right)^{n-1}}\right)\cdot
\nabla_{\mathbb{S}^{n}} \sin{d}\,
d\sigma_{n}\\
&=
\nonumber
\int_{\mathbb{S}^{n}}
\frac{-n|u|^{n-2}u \,\cos{d}\,\nabla_{\mathbb{S}^{n}} u
\cdot \nabla_{\mathbb{S}^{n}}{d}}
{\sin^{n-1}{d}\left(
\log{\frac{e}{\sin{d}}}\right)^{n-1}}
d\sigma_{n}+(n-1)\int_{\mathbb{S}^{n}}
\frac{|u|^n\cos^{2}{d}}{\sin^{n}{d}\left(
\log{\frac{e}{\sin{d}}}\right)^{n-1}}
d\sigma_{n}\\&
\label{prtsfc1}
\;\;\;-(n-1)\int_{\mathbb{S}^{n}}
\frac{|u|^n}{\sin^{n}{d}\left(
\log{\frac{e}{\sin{d}}}\right)^{n}}\cos^{2}{d}\,
d\sigma_{n}.
\end{align}
Using H\"{o}lder's inequality then applying Young's inequality implies
\begin{align}
\nonumber
&\int_{\mathbb{S}^{n}}
\frac{-n|u|^{n-2}u \,\cos{d}\,\nabla_{\mathbb{S}^{n}} u
\cdot \nabla_{\mathbb{S}^{n}}{d}}
{\sin^{n-1}{d}\left(
\log{\frac{e}{\sin{d}}}\right)^{n-1}}d\sigma_{n}\\
&\label{yngf1f1fc1} \leq
(n-1)\beta^{\frac{n}{n-1}}\int_{\mathbb{S}^{n}}
\frac{|u|^{n}}
{\sin^{n}{d}\left(
\log{\frac{e}{\sin{d}}}\right)^{n}}
d\sigma_{n}+
\frac{1}{\beta^{n}}
\int_{\mathbb{S}^{n}}
|\nabla_{\mathbb{S}^{n}} u|^{n}|\cos{d}|^{n}
d\sigma_{n},
\end{align}
with $\beta>0$. Returning with the estimate (\ref{yngf1f1fc1}) to the inequality (\ref{prtsfc1}), we deduce that
\begin{equation}\label{prtsf1fnfnc2}
\begin{split}
&(n-1)\left(1-\beta^{\frac{n}{n-1}}\right)\beta^{n}
\int_{\mathbb{S}^{n}}
\frac{|u|^n}{\sin^{n}{d}\left(
\log{\frac{e}{\sin{d}}}\right)^{n}}
d\sigma_{n}
\\
&\leq\int_{\mathbb{S}^{n}}
|\nabla_{\mathbb{S}^{n}} u|^{n}|\cos{d}|^{n}
d\sigma_{n}+\beta^{n}n
\int_{\mathbb{S}^{n}}\frac{|u|^n}{\sin^{n-2}{d}\left(
\log{\frac{e}{\sin{d}}}\right)^{n-1}}
d\sigma_{n},
\end{split}
\end{equation}
where we estimated $\displaystyle
\int_{\mathbb{S}^{n}}\frac{|u|^n}{\sin^{n-2}{d}\left(
\log{\frac{e}{\sin{d}}}\right)^{n}}\leq \int_{\mathbb{S}^{n}}\frac{|u|^n}{\sin^{n-2}{d}\left(
\log{\frac{e}{\sin{d}}}\right)^{n-1}}$.
The value  $\displaystyle \beta=\left(\frac{n-1}{n}\right)^{\frac{n-1}{n}}$
optimizes (\ref{prtsf1fnfnc2}) and produces (\ref{fc1}). \newline
\indent Again, the inequality (\ref{fc2}) can be proved in the same way as (\ref{f3}).
Since
$\Delta_{\mathbb{S}^{n}}d=(n-1)/\tan{d}$, then
invoking the divergence theorem we get
\begin{align}
\nonumber &(n-1)\int_{\mathbb{S}^{n}}\frac{|u|^{n}}{|\tan{d}|^{n}
\left(\log{\frac{e}{\sin{d}}}\right)^{n-1}}
d\sigma_{n}\\
\nonumber&=
\int_{\mathbb{S}^{n}}
\frac{|u|^{n}}{|\tan{d}|^{n-2}\tan{d}
\left(\log{\frac{e}{\sin{d}}}\right)^{n-1}}
\Delta_{\mathbb{S}^{n}} d\, d\sigma_{n}\\
\nonumber
&=-
\int_{\mathbb{S}^{n}}\nabla_{\mathbb{S}^{n}} d\cdot
\nabla_{\mathbb{S}^{n}} \left(
\frac{|u|^{n}}{|\tan{d}|^{n-2}
\tan{d}\left(\log{\frac{e}{\sin{d}}}\right)^{n-1}}\right)
d\sigma_{n}\\
&=
\nonumber
\int_{\mathbb{S}^{n}}
\frac{-n|u|^{n-2}u \nabla_{\mathbb{S}^{n}} u
\cdot \nabla_{\mathbb{S}^{n}} d}
{|\tan{d}|^{n-2}\tan{d}
\left(\log{\frac{e}{\sin{d}}}\right)^{n-1}}
d\sigma_{n}+\\
\nonumber&\;\;\;+(n-1)\int_{\mathbb{S}^{n}}
\frac{|u|^{n}}
{|\tan{d}|^{n}\left(\log{\frac{e}{\sin{d}}}
\right)^{n-1}}\frac{1}{\cos^{2}{d}}
d\sigma_{n}+\\
\label{prts2}
&\;\;\;-(n-1)\int_{\mathbb{S}^{n}}
\frac{|u|^{n}}
{|\tan{d}|^{n}\left(\log{\frac{e}{\sin{d}}}\right)^{n}}
d\sigma_{n}.
\end{align}
Analogously to the inequality (\ref{yng}), utilizing
H\"{o}lder's inequality followed by Young's inequality we obtain
\begin{align}
\nonumber
&\int_{\mathbb{S}^{n}}
\frac{-n|u|^{n-2}u \nabla_{\mathbb{S}^{n}} u
\cdot \nabla_{\mathbb{S}^{n}} d}
{|\tan{d}|^{n-2}\tan{d}
\left(\log{\frac{e}{\sin{d}}}\right)^{n-1}}
d\sigma_{n}\\
& \label{yng2}
\leq (n-1)\beta^{\frac{n}{n-1}}\int_{\mathbb{S}^{n}}
\frac{|u|^{n}}
{|\tan{d}|^{n}\left(\log{\frac{e}{\sin{d}}}\right)^{n}}
d\sigma_{n}+\frac{1}{\beta^{n}}
\int_{\mathbb{S}^{n}}
|\nabla_{\mathbb{S}^{n}} u|^{n}
d\sigma_{n},
\end{align}
for any $\beta>0$.
Inserting the estimate (\ref{yng2}) into  (\ref{prts2})
and rearranging while taking into account the identity $\sec^{2}{d}=1+\tan^{2}{d}$ produces
\begin{align}
\nonumber &\beta^{n}\left(1-\beta^{\frac{n}{n-1}}\right)
(n-1)\int_{\mathbb{S}^{n}}
\frac{|u|^{n}}{|\tan{d}|^{n}\left(\log{\frac{e}{\sin{d}}}
\right)^{n}}d\sigma_{n}\\
\label{f4}&\hspace*{-1cm}\leq \int_{\mathbb{S}^{n}}
|\nabla_{\mathbb{S}^{n}} u|^{n}
d\sigma_{n}+\beta^{n}(n-1)\int_{\mathbb{S}^{n}}
\frac{|u|^{n}\,d\sigma_{n}}
{|\tan{d}|^{n-2}\left(\log{\frac{e}{\sin{d}}}
\right)^{n-1}}.
\end{align}
The value of $\beta$ that optimizes the inequality (\ref{f4}) is
$\displaystyle \beta=\left(\frac{n-1}{n}\right)^{\frac{n-1}{n}}$.
\end{proof}
\begin{thm}(Sharpness of the inequalities (\ref{fc1})-(\ref{fc2}))\\
The inequality (\ref{fc1}) is optimal is the following sense:
\begin{align}
\label{fc1shrp1}&\sup_{u\in W^{1,n}\left(\mathbb{S}^{n}\right)\setminus\left\{ 0\right\}}
{\frac{\gamma^{n}_{n}U_{n}(u)}{Q_{n}(u) +
n\gamma^{n-1}_{n}\widetilde{U}_{n}(u)}}=1,\\
\label{fc1shrp2}&\sup_{u\in W^{1,n}\left(\mathbb{S}^{n}\right)\setminus\left\{ 0\right\}}
{\frac{\gamma_{n}U_{n}(u)-n
\widetilde{U}_{n}(u)}{Q_{n}(u)}}=\gamma^{1-n}_{n}.
%\label{fc1shrp3}&\sup_{u\in W^{1,n}\left(\mathbb{S}^{n}\right)\setminus\left\{ 0\right\}}
%{\frac{\gamma^{n}_{n}U_{n}(u)-Q_{n}(u)}{
%\widetilde{U}_{n}(u)}}=n\gamma^{n-1}_{n}.
\end{align}
The inequality (\ref{fc2}) is also sharp. We precisely have
\begin{eqnarray}
\label{fc2shrp1}\sup_{u\in W^{1,n}\left(\mathbb{S}^{n}\right)\setminus\left\{ 0\right\}}
{\frac{\gamma^{n}_{n}V_{n}(u)}{H_{n}(u) +
(n-1)\gamma^{n-1}_{n}\widetilde{V}_{n}(u)}}&=&1,\\
\label{fc2shrp2}\sup_{u\in W^{1,n}\left(\mathbb{S}^{n}\right)\setminus\left\{ 0\right\}}
{\frac{\gamma_{n}V_{n}(u)-(n-1)
\widetilde{V}_{n}(u)}{H_{n}(u)}}&=&\gamma^{1-n}_{n},\\
\label{fc2shrp3}\sup_{u\in W^{1,n}\left(\mathbb{S}^{n}\right)\setminus\left\{ 0\right\}}
{\frac{\gamma^{n}_{n}V_{n}(u)-H_{n}(u)}{
\widetilde{V}_{n}(u)}}&=&(n-1)\gamma^{n-1}_{n}.
\end{eqnarray}
\end{thm}
\begin{proof}
We prove (\ref{fc1shrp1}) - (\ref{fc2shrp3}) via introducing a sequence of nonzero real functions in
$W^{1,n}\left(\mathbb{S}^{n}\right)$ for which the respective suprema are attained. Define
\begin{equation*}
f_{\epsilon}(\Theta_{n}):=
\left(\log{\frac{e}{\sin{d(\Theta_{n},\Phi_{n})}}}
\right)^{\frac{n-1-\epsilon}{n}},\;
\Theta_{n}\in \mathbb{S}^{n}\setminus
\left\{\pm \Phi_{n}\right\}.
\end{equation*}
Then, using formula (\ref{frml}), we have
\begin{align}
\nonumber&\hspace*{-1cm}\int_{\mathbb{S}^{n}}
\frac{|f_{\epsilon}|^{n}d\sigma_{n}}{\sin^{n}{d}\left(
\log{\frac{e}{\sin{d}}}\right)^{n}}=
\int_{\mathbb{S}^{n}}
\frac{d\sigma_{n}}{\sin^{n}{d}\left(
\log{\frac{e}{\sin{d}}}\right)^{1+\epsilon}}=2C_{n}
\int_{0}^{1}\frac{ds}{(1-s^2)
\left(\log{\frac{e}{\sqrt{1-s^2}}}\right)^{1+\epsilon}}\\
&\label{fc1s1}
=2C_{n}\int_{0}^{1}\frac{ds}{s\sqrt{1-s^2}
\left(\log{\frac{e}{s}}\right)^{1+\epsilon}}
=2C_{n}\left(I(\epsilon)+\tilde{I}(\epsilon)\right),
\end{align}
where
\begin{equation}\label{ispln}
I(\epsilon):=\int_{0}^{1}\frac{ds}{s
\left(\log{\frac{e}{s}}\right)^{1+\epsilon}}=
\frac{1}{\epsilon},
\end{equation}
\begin{equation}\label{isplntld}
\begin{split}
0<\tilde{I}(\epsilon):&=\int_{0}^{1}
\left(\frac{1}{\sqrt{1-s^2}}-1\right)\frac{ds}{s
\left(\log{\frac{e}{s}}\right)^{1+\epsilon}}\\
&=\int_{0}^{1}
\frac{sds}{\sqrt{1-s^2}\left(1+\sqrt{1-s^2}\right)
\left(\log{\frac{e}{s}}\right)^{1+\epsilon}}\leq
\int_{0}^{1}
\frac{sds}{\sqrt{1-s^2}}=1
\end{split}
\end{equation}
uniformly in $\epsilon$. Moreover,
\begin{align}
\nonumber&
\frac{\int_{\mathbb{S}^{n}}
\left|\nabla_{\mathbb{S}^{n}}
f_{\epsilon}\right|^{n}\left|\cos{d}\right|^{n}
d\sigma_{n}}{\left(\gamma_{n}-\frac{1}{n}\epsilon\right)^{n}}=
\int_{\mathbb{S}^{n}}
\frac{|\cos{d}|^{2n}}{\sin^{n}{d}\left(
\log{\frac{e}{\sin{d}}}\right)^{1+\epsilon}}d\sigma_{n}
\\ \nonumber &=2C_{n}
\int_{0}^{1}\frac{s^{2n}}{(1-s^2)
\left(\log{\frac{e}{\sqrt{1-s^2}}}\right)^{1+\epsilon}}ds
=2C_{n}
\int_{0}^{1}\frac{\left(1-s^2\right)^{
\frac{2n-1}{2}}}{s
\left(\log{\frac{e}{s}}\right)^{1+\epsilon}}ds
\\ \label{fc1s2}&=
2C_{n}\left(I(\epsilon)-\bar{I}(\epsilon)\right),
\end{align}
with
\begin{align}
\nonumber &\hspace*{-1 cm}
0<\bar{I}(\epsilon):=\int_{0}^{1}
\frac{1-\left(1-s^2\right)^{
\frac{2n-1}{2}}}{s
\left(\log{\frac{e}{s}}\right)^{1+\epsilon}}ds=
\int_{0}^{1}
\frac{\left(1-\left(1-s^2\right)^{2n-1}\right)ds}{
\left(1+\left(1-s^2\right)^{
\frac{2n-1}{2}}\right)s
\left(\log{\frac{e}{s}}\right)^{1+\epsilon}}\\
\label{ibrspln} &\leq
\int_{0}^{1}
\frac{1-\left(1-s^2\right)^{2n-1}}{s}ds
=\sum_{r=1}^{2n-1}\binom{2n-1}{r}
\frac{(-1)^{r+1}}{2r+1}.
\end{align}
It follows from (\ref{ibrspln}) that
\begin{equation}\label{ibrspln01}
\lim_{\epsilon\rightarrow 0^{+}}\bar{I}(\epsilon)
=\bar{I}(0)=O(1)
\end{equation}
with an implicit constant that depends only on
$n$. Furthermore
\begin{align}
\nonumber&\int_{\mathbb{S}^{n}}
\frac{|f_{\epsilon}|^{n}d\sigma_{n}}{\sin^{n-2}{d}\left(
\log{\frac{e}{\sin{d}}}\right)^{n}}=
\int_{\mathbb{S}^{n}}
\frac{d\sigma_{n}}{\sin^{n-2}{d}\left(
\log{\frac{e}{\sin{d}}}\right)^{1+\epsilon}}\\
&\label{fc1s3}
=2C_{n}
\int_{0}^{1}\frac{ds}{
\left(\log{\frac{e}{\sqrt{1-s^2}}}\right)^{1+\epsilon}}
\longrightarrow 2C_{n}
\int_{0}^{1}\frac{ds}{\log{\frac{e}{\sqrt{1-s^2}}}},
\end{align}
as $\epsilon\rightarrow 0^{+}$, by the dominated (or monotone) convergence theorem. Observe that $\int_{0}^{1}\frac{ds}{
\log{\frac{e}{\sqrt{1-s^2}}}}\leq 1$. Using
(\ref{fc1s1}) together with (\ref{ispln}), and (\ref{fc1s2}), we obtain that
\begin{equation*}
\begin{split}
&\frac{\gamma^{n}_{n}U_{n}(f_{\epsilon})}{Q_{n}(f_{\epsilon}) +n\gamma^{n-1}_{n}\widetilde{U}_{n}(f_{\epsilon})}=
\frac{\gamma^{n}_{n}\left(\frac{1}{\epsilon}+
\tilde{I}(\epsilon)\right)}{
\left(\gamma_{n}-\frac{1}{n}\epsilon\right)^{n}
\left(\frac{1}{\epsilon}-
\bar{I}(\epsilon)\right) +n\gamma^{n-1}_{n}\frac{\widetilde{U}_{n}(f_{\epsilon})}
{2C_{n}}}\\
&=\frac{\gamma^{n}_{n}\left(1+
\epsilon\tilde{I}(\epsilon)\right)}{
\left(\gamma_{n}-\frac{1}{n}\epsilon\right)^{n}
\left(1-
\epsilon\bar{I}(\epsilon)\right) +n\gamma^{n-1}_{n}\epsilon\,
\frac{\widetilde{U}_{n}(f_{\epsilon})}{2C_{n}}}
\longrightarrow 1
\end{split}
\end{equation*}
by the limits (\ref{ibrspln01}) and (\ref{fc1s3}).
This convergence proves (\ref{fc1shrp1}). In the same manner
\begin{equation*}
\begin{split}
&\frac{\gamma^{n}_{n}U_{n}(f_{\epsilon})-
n\gamma^{n-1}_{n}\widetilde{U}_{n}(f_{\epsilon})}{
Q_{n}(f_{\epsilon})}=
\frac{\gamma^{n}_{n}\left(1+
\epsilon\tilde{I}(\epsilon)\right)-
n\gamma^{n-1}_{n}\epsilon\,
\frac{\widetilde{U}_{n}(f_{\epsilon})}{2C_{n}}}{
\left(\gamma_{n}-\frac{1}{n}\epsilon\right)^{n}
\left(1-
\epsilon\bar{I}(\epsilon)\right)}
\longrightarrow 1,
\end{split}
\end{equation*}
which proves (\ref{fc1shrp2}).

Proceeding, we have
\begin{equation}\label{fc2s1}
\int_{\mathbb{S}^{n}}
\frac{|f_{\epsilon}|^{n}d\sigma_{n}}{|\tan{d}|^{n}\left(
\log{\frac{e}{\sin{d}}}\right)^{n}}=
 J_{n}(\epsilon),\;
 \int_{\mathbb{S}^{n}}
\left|\nabla_{\mathbb{S}^{n}}
f_{\epsilon}\right|^{n}d\sigma_{n}=
\left(\gamma_{n}-\frac{1}{n}\epsilon\right)^{n}
 J_{n}(\epsilon),
\end{equation}
where
\begin{equation}\label{fc2s2}
\begin{split}
&J_{n}(\epsilon):=
\int_{\mathbb{S}^{n}}
\frac{d\sigma_{n}}{|\tan{d}|^{n}\left(
\log{\frac{e}{\sin{d}}}\right)^{1+\epsilon}}=2C_{n}
\int_{0}^{1}\frac{s^{n}\,ds}{(1-s^2)
\left(\log{\frac{e}{\sqrt{1-s^2}}}\right)^{1+\epsilon}}\\
&=2C_{n}\int_{0}^{1}\frac{(1-s^2)^{\frac{n-1}{2}}\,ds}{
s\left(\log{\frac{e}{s}}\right)^{1+\epsilon}}=
2C_{n}\left(I(\epsilon)+
\tilde{J}_{n}(\epsilon)\right)
=2C_{n}\left(\frac{1}{\epsilon}+\tilde{J}_{n}(\epsilon)
\right),
\end{split}
\end{equation}
where
\begin{equation*}
\tilde{J}_{n}(\epsilon):=
\int_{0}^{1}\frac{(1-s^2)^{\frac{n-1}{2}}-1}{
s\left(\log{\frac{e}{s}}\right)^{1+\epsilon}}ds.
\end{equation*}
Notice here that
\begin{equation}\label{implct}
\lim_{\epsilon\rightarrow 0^{+}}\tilde{J}_{n}(\epsilon)=\tilde{J}_{n}(0)=O(1),
\end{equation}
where the implicit constant depends solely on the dimension $n$. This follows from
the dominated convergence theorem as we have the uniform bound
\begin{equation*}
\begin{split}
&|\tilde{J}_{n}(\epsilon)|\leq
\int_{0}^{1}\frac{1-(1-s^2)^{n-1}}{
s\left(1+(1-s^2)^{\frac{n-1}{2}}\right)
\left(\log{\frac{e}{s}}\right)^{1+\epsilon}}ds
\leq\int_{0}^{1}\frac{1-(1-s^2)^{n-1}}{
s}ds\\
&=\frac{1}{2}\sum_{r=1}^{n-1}\binom{n-1}{r}
\frac{(-1)^{r+1}}{r}.
\end{split}
\end{equation*}
We also have
\begin{equation}\label{fc2s3}
\int_{\mathbb{S}^{n}}
\frac{|f_{\epsilon}|^{n}\,d\sigma_{n}}{|\tan{d}|^{n-2}\left(
\log{\frac{e}{\sin{d}}}\right)^{n-1}}=2C_{n}
\int_{0}^{1}\frac{s^{n-2}\,ds}{
\left(\log{\frac{e}{\sqrt{1-s^2}}}\right)^{\epsilon}}
\longrightarrow \frac{2C_{n}}{n-1},
\end{equation}
by the dominated convergence theorem. Now,
using (\ref{fc2s1}) and (\ref{fc2s2})
we get
\begin{equation*}
\begin{split}
&\frac{\gamma^{n}_{n}V_{n}(f_{\epsilon})}{H_{n}(f_{\epsilon}) +(n-1)\gamma^{n-1}_{n}\widetilde{V}_{n}(f_{\epsilon})}\\
&=\frac{\gamma^{n}_{n}(1+\epsilon \tilde{J}_{n}(\epsilon))}{
\left(\gamma_{n}-\frac{\epsilon}{n}\right)^{n}(1+\epsilon \tilde{J}_{n}(\epsilon)) +
(n-1)\gamma^{n-1}_{n}\epsilon
\frac{\widetilde{V}_{n}(f_{\epsilon})}{2C_{n}}}
\longrightarrow 1
\end{split}
\end{equation*}
as $\epsilon\longrightarrow 0^{+}$ by (\ref{implct}) and the convergence in (\ref{fc2s3}). This proves (\ref{fc2shrp1}). We also have
\begin{equation*}
\begin{split}
&\frac{\gamma_{n}V_{n}(f_{\epsilon})-
(n-1)\widetilde{V}_{n}(f_{\epsilon})}{
H_{n}(f_{\epsilon}) }\\
&=\frac{\gamma_{n}(1+\epsilon \tilde{J}_{n}(\epsilon)))-
(n-1)\epsilon\frac{\widetilde{V}_{n}(f_{\epsilon})}{2C_{n}
}}{\left(\gamma_{n}-\frac{\epsilon}{n}\right)^{n}(1+\epsilon \tilde{J}_{n}(\epsilon))}
\longrightarrow \frac{1}{\gamma^{n-1}_{n}}.
\end{split}
\end{equation*}
when $\epsilon\longrightarrow 0^{+}$ using (\ref{implct}) and (\ref{fc2s3}). This proves (\ref{fc2shrp2}). Finally, we prove (\ref{fc2shrp3}). Employing (\ref{fc2s1}) and (\ref{fc2s2}) one last time, we find
\begin{equation*}
\begin{split}
&\frac{\gamma^{n}_{n}V_{n}(f_{\epsilon})-
H_{n}(f_{\epsilon})}{
\widetilde{V}_{n}(f_{\epsilon})}
=\frac{\left(\gamma^{n}_{n}-
\left(\gamma_{n}-\frac{\epsilon}{n}\right)^{n}\right)
(\frac{1}{\epsilon}+ \tilde{J}_{n}(\epsilon))}{
\frac{\widetilde{V}_{n}(f_{\epsilon})}{2C_{n}}}\\
&=\frac{\left(\gamma^{n}_{n}-
\left(\gamma_{n}-\frac{\epsilon}{n}\right)^{n}\right)
}{\epsilon}
\frac{(1+\epsilon \tilde{J}_{n}(\epsilon))}{\frac{
\widetilde{V}_{n}(f_{\epsilon})}{2C_{n}}}
\longrightarrow (n-1)\gamma^{n-1}_{n}
\end{split}
\end{equation*}
by (\ref{implct}) and the limit in (\ref{fc2s3}).

\end{proof}

\end{document}